\documentclass[12pt]{article}

\usepackage{latexsym}
\usepackage{amssymb}
\usepackage{graphicx}
\usepackage{color}
\usepackage{amsmath}
\NewDocumentCommand{\evalat}{sO{\big}mm}{%
  \IfBooleanTF{#1}
   {\mleft. #3 \mright|_{#4}}
   {#3#2|_{#4}}%
}
\usepackage{tikz}
\usepackage{scalerel}
\usepackage{pict2e}
\usepackage{tkz-euclide}
\usepackage{pgfplots}
\usetikzlibrary{patterns, arrows.meta}
\usetikzlibrary{shadows}
\usetikzlibrary{external}
\usetikzlibrary{intersections,calc}
\usepackage{pgfplots}
\usepackage{caption}
\pgfplotsset{compat=1.18}
\pgfplotsset{compat=newest}
\usepgfplotslibrary{fillbetween}
\usepgfplotslibrary{statistics}
\usepackage{xcolor}

\usepackage{mathrsfs}
\usepackage{comment}

\usepackage{helvet}

\DeclareMathSizes{12}{12}{10}{10}

\newtheorem{theorem}{Theorem}

\newtheorem{corollary}[theorem]{Corollary}

\newtheorem{definition}[theorem]{Definition}

\newtheorem{lemma}[theorem]{Lemma}

\newenvironment{proof}[1][Proof]{\textbf{#1.} }{\ \rule{0.5em}{0.5em}}
\newcommand{\kom}[1]{}
\renewcommand{\kom}[1]{{\bf [#1]}}
\definecolor{blau}{rgb}{0.1,0.0,0.9}

\newcounter{komcounter}

\numberwithin{komcounter}{section}

\newcommand{\grad}{\nabla}
\title{Estimates for viscosity solutions of fully nonlinear equations near smooth boundaries}

\author{Niklas L.P. Lundstr\"om{\small{$^1$}}, Marcus Olofsson{\small{$^2$}},
Jesper Singh{\small{$^1$}}
}
\date{\today}

\begin{document}

\maketitle

\begin{center}
{\small{$^1$}}
{\it \small Department of Mathematics and Mathematical Statistics},\\
{\it \small Ume{\aa} University, SE-90187 Ume{\aa}, Sweden\/{\rm ;}}\\
{\it \small niklas.lundstrom@umu.se; jesper.singh@umu.se}\\
{\small{$^2$}}
{\it \small Odensalaskolan, Östersund, Sweden; bakkenolofsson@gmail.com}\\   
\end{center}

\begin{abstract}
We reduce the problem of proving decay estimates for positive viscosity solutions of fully nonlinear PDEs to proving analogous estimates for solutions of one-dimensional ordinary differential inequalities. Our machinery allows the ellipticity to vanish near the boundary and permits general, possibly unbounded, lower-order terms. A key consequence is the derivation of boundary Harnack inequalities for a broad class of fully nonlinear, nonhomogeneous equations near $C^{1,1}$-boundaries.

In combination with  $C^{1,\alpha}$-estimates, 
we also find that quotients of positive solutions vanishing on a boundary portion are Hölder continuous near $C^{1,1}$-boundaries. This result applies to a wide family of fully nonlinear uniformly elliptic PDEs; and to $p(x)$-harmonic functions and planar $\infty$-harmonic functions near locally flat boundaries. 
We end by deriving Phragm\'en--Lindel\"of-type corollaries in unbounded domains.\\


\end{abstract}

%%%%%%%%%%%%%%%%%%%%%%%%%%%%%%%%%%%%%%%%%%%% 
%%%%%%%%%%%%%%%%%%%%%%%%%%%%%%%%%%%%%%%%%%%%
%%%%%%%%%%%%%%%%%%%%%%%%%%%%%%%%%%%%%%%%%%%%
%%%%%%%%%%%%%%%%%%%%%%%%%%%%%%%%%%%%%%%%%%%%
%%%%%%%%%%%%%%%%%%%%%%%%%%%%%%%%%%%%%%%%%%%%
%%%%%%%%%%%%%%%%%%%%%%%%%%%%%%%%%%%%%%%%%%%%

\section{Introduction}

\setcounter{theorem}{0}
\setcounter{equation}{0}

%%%%%%%%%%%%%%%%%%%%%%%%%%%%%
% PART I - DECAUY ESTIMATES %
%%%%%%%%%%%%%%%%%%%%%%%%%%%%%

\noindent
In the first part of the paper, we develop a method that reduces the problem of obtaining decay estimates near the boundary for viscosity sub- and supersolutions of fully nonlinear PDEs to proving analogous estimates for classical solutions of one-dimensional problems in the form of ordinary differential inequalities (ODIs). These ODIs are only slight modifications of the original PDEs in one dimension, thereby illustrating a simple connection between the 
$n$-dimensional and the one-dimensional setting for this class of problems.
We prove these results in Section \ref{sec:decay}, in which Theorem \ref{th:lower} provides our lower decay estimates of positive viscosity supersolutions $v$ of the PDE \eqref{eq:main-nonlinear}, 
being nonnegative on a portion of the boundary of a domain satisfying an interior sphere condition. 
The estimate takes the form $v(x) \geq \check h(d(x,\partial \Omega))$, in which $\check h$ is a classical solution of the corresponding one-dimensional ODI \eqref{eq:ODE}. %satisfying certain boundary conditions.
Analogously, Theorem \ref{th:upper} gives an upper estimate $u(x) \leq \hat h(d(x,\partial \Omega))$ for viscosity subsolutions $u$ of \eqref{eq:main-nonlinear}, 
being nonpositive on a portion of the boundary of a domain satisfying an exterior sphere condition, 
where $\hat h$ solves the ODI \eqref{eq:ODE2}. %satisfying certain boundary conditions.
In the next step we conclude the following similar characterization of the boundary Harnack inequality near $C^{1,1}$ boundaries: 
Suppose that $u_1$ and $u_2$ are viscosity solutions of \eqref{eq:main-nonlinear}, 
vanishing continuously on a portion of the boundary. 
Then, if the boundary Harnack inequality holds for the solutions of the corresponding ODIs, then it holds also for $u_1$ and $u_2$.
It thus suffices to prove the inequality for some suitable solutions of the one-dimensional ODIs, see Corollary \ref{cor:BHI}. 
To establish the above mentioned results we only assume our main operator in \eqref{eq:main-nonlinear} to be proper, satisfy \eqref{eq:ass_drift_sub} and/or \eqref{eq:ass_drift_super}.

We believe that our work extends the existing literature in several notable ways:
(i) We allow the ellipticity to vanish near the boundary -- an explicit example is provided in Section \ref{sec:example_0_vanish}.
(ii) We consider general unbounded lower-order terms that may depend on the spatial variable, the solution itself, and the norm of its gradient.
Overall, we impose minimal restrictions on the PDE. Instead, our methodology centers on what can be proved for the associated ODIs:
(iii) Rather than verifying whether a given PDE satisfies a set of technical assumptions, 
our approach allows one to reduce the problem to analyzing one-dimensional differential equations. 
Often, this reduces to classical textbook examples; for instance, see Section \ref{sec:example_0}, 
which includes generalizations of the well-known $p$-Laplace equation.
The present approach allows us to explicitly track how estimates depend on the value of the solution at a given distance from the boundary. 
This is demonstrated by examples in Sections \ref{sec:example_-1,1} and \ref{sec:example_pk}, where we identify settings in which decay estimates and boundary Harnack inequalities can be derived for non-homogeneous PDEs that clearly fall outside the scope of assumptions made in, for example, \cite{LOT20}.

In the setting of harmonic functions the Boundary Harnack inequality was considered already in the 70's. %, e.g.,  \cite{K72, An78, D77, W78}. 
Generalizations to elliptic equations, nontangentially accessible domains, H\"older domains and uniform domains appeared in the 80's. %can be found in \cite{CFMS81, JK82, BBB91, BB91, Ai01,SS20}.
Concerning nonlinear PDEs, similar results were obtained for $p$-harmonic functions in smooth domains \cite{aikawa}, and in Lipschitz, and Reifenberg flat domains \cite{LN07, LN10, LN12a, LN12b}. 
In the latter, the ratio $u/v$ was also proved to be H\"older continuous, 
and applications of such regularity estimates to free boundary problems and 
studies of the Martin boundary can be found in \cite{LN10, LN12a, LN12b}. 
For recent results on boundary Harnack inequalities in Lipschitz domains and beyond, as well as applications to free boundary problems, see e.g. \cite{TTV10, KS11, AS19,DS20,R-OT-L21,AKS22}, and for nonhomogeneous equations such as the variable $p$-Laplace equation, we refer to \cite{AL16, AJ17, LOT20}. For nonlocal elliptic equations in non-divergence form, see e.g. \cite{R-OS19} and references therein. 
Finally, we note that our approach is inspired by our earlier work \cite{L22} on Phragm\'en--Lindel\"of-type estimates, and that, in contrast to the aforementioned studies, our contribution mainly relies on less restrictive assumptions on the equation.

%%%%%%%%%%%%%%%%%%%%%%%%%%%%%%
% PART 2 - REGULARITY ON u/v %
%%%%%%%%%%%%%%%%%%%%%%%%%%%%%%

In the second part of the paper (Section \ref{sec:holder})  
we first show that the ratio $u_1/u_2$ is Hölder continuous, where $u_1$ and $u_2$ are positive solutions of a fully nonlinear PDE, 
vanishing on a portion of the boundary of a $C^{1,1}$-domain, see Theorem \ref{th:u/v Hölder fully nonlinear}. 
In particular, 
by combining our results from the first part of the paper, 
giving bounds on the ratios $u_i/d(x,\partial \Omega)$, $d(x,\partial \Omega)/u_i$, $i = 1,2$, and $u_1/u_2$ near the boundary, with known $C^{1,\alpha}$-estimates of $u_i$, 
we prove that these ratios are also Hölder continuous of order $\alpha$.
To apply the most general available global $C^{1,\alpha}$-estimates in our setting -- which, to the best of our knowledge, are given in~\cite{N19} -- we had to impose uniformly ellipticity and at most quadratic growth in the gradient, see assumption \eqref{eq:ass-nonlinear1}. 

Secondly, we consider equations for which global $C^{1,\alpha}$-estimates are unknown, and use reflection arguments together with local $C^{1,\alpha}$-estimates to establish the Hölder continuity of the ratios. 
In particular, we prove results for positive $p(x)$-harmonic functions (Theorem \ref{th:p(x)-harmonic}), %in $\mathbb{R}_+^n$, 
and for planar $\infty$-harmonic functions (Theorem \ref{th:holder-boundary-inf-plane}), %in $\mathbb{R}_+^2$.  
vanishing on a hyperplane.
A consequence of the latter is a gradient estimate of the form $|D u| \approx u(x) / d(x,\partial  \Omega)$, see Corollary \ref{cor:grad-inf-est}.

We end the paper with some further corollaries of our results, as well as corollaries of already known versions of these estimates for $p$-harmonic functions, $p \in (1,\infty)$ in Section \ref{sec:applications-of-u/v-reg}. In particular, we derive some Phragm\'en--Lindel\"of theorems as well as uniqueness (modulo normalization) of infinity- and $p$-harmonic functions in some unbounded domains.  

\bigskip

\noindent
{\bf Notation.}
Points in Euclidean $n$-space $\mathbb{R}^n$ are denoted by $x=(x_1,...,x_n)$ or $x=(x',x_n)$, where $x'=(x_1,...,x_{n-1})\in\mathbb{R}^{n-1}$, 
and $\mathbb R^n_+:=\{x \in \mathbb R^n: x_n>0\}$. 
We let $\overline{E}$ be the closure, $\complement E$ the complement, and $\partial E$ the boundary of a set $E \subset \mathbb{R}^n$,  
and $d(y, E)$ equals the distance from $y \in \mathbb{R}^n$ to $E$. 
The space of continuous functions on $E$ is denoted by $C(E)$; 
$C^{\alpha}(E)$, for $0<\alpha\leq 1$, consists of functions in $C(E)$ that are Hölder continuous of order $\alpha$,
$C^k(E)$ is the space of functions whose partial derivatives of order $k$ exist and are continuous on $E$, 
and $C^{1,\alpha}$ consists of functions in $C^1(E)$  whose first-order derivatives are also Hölder continuous of order $\alpha$. 
$LSC(E)$ denotes lower, while $USC(E)$ denotes upper, semicontinuous functions on $E$. 
$\langle .,.\rangle$ denotes the standard Euclidean inner product on $\mathbb{R}^n$ and $|x|=\sqrt{\langle x, x \rangle}$ is the Euclidean norm of $x$. The open ball of $x$ with radius $r>0$ is denoted by $B(x,r)=\{y\in\mathbb{R}^n:|x-y|<r\}$,  
and $dx$ is the Lebesgue $n$-measure on $\mathbb{R}^n$. 
As a convention, in this paper, unless otherwise
stated, $c$ will denote a constant $\geq 1$, not necessarily the same at each occurrence. 
We write $A \lesssim B$ ($A \gtrsim B$) if there exists $c$ such that $A \leq c B$ ($c A \geq B$), 
and $A \approx B$ means there exists $c$ so that $c^{-1} A \leq B \leq c A$. 
By $Du$ we denote a `weak' gradient, $D^2u$ a `weak' hessian, 
and by $\mathbb{S}^n$ we denote the set of symmetric $n \times n$ matrices equipped with the positive semi-definite ordering;
for $X, Y \in \mathbb{S}^n$, we write $X \leq Y$ if $\langle (X - Y) \xi, \xi \rangle \leq 0$ for all $\xi \in \mathbb{R}^n$.
Finally, a connected open set will be referred to as a domain.

\tableofcontents

%%%%%%%%%%%%%%%%%%%%%%%%%%%%%%%%%%%%%%%%%%%%
%%%%%%%%%%%%%%%%%%%%%%%%%%%%%%%%%%%%%%%%%%%%
%%%%%%%%%%%%%%%%%%%%%%%%%%%%%%%%%%%%%%%%%%%%
%%%%%%%%%%%%%%%%%%%%%%%%%%%%%%%%%%%%%%%%%%%%
%%%%%%%%%%%%%%%%%%%%%%%%%%%%%%%%%%%%%%%%%%%%
%%%%%%%%%%%%%%%%%%%%%%%%%%%%%%%%%%%%%%%%%%%%

\section{Decay estimates near the boundary}
\label{sec:decay}

\setcounter{theorem}{0}
\setcounter{equation}{0}

We consider fully nonlinear degenerate elliptic equations in non-divergence form,
\begin{align}\label{eq:main-nonlinear}
F(x,u,Du,D^2u) = 0,%\tag{$\star$}
\end{align}
where 
$F:\mathbb{R}^n \times \mathbb{R} \times \mathbb{R}^n  \times \mathbb{S}^n \rightarrow \mathbb{R}$.
We will assume throughout that
\begin{align}\label{eq:deg-ellipt-proper}
F(x,u,p,X) \geq  F(x,v,p,Y) \quad \text{whenever} \quad u \geq v \quad \text{and} \quad X \leq Y,
\end{align}
and adopt the standard concept of viscosity solutions:
A $USC$ function $u : \Omega \to \mathbb{R}$ is a viscosity subsolution
if for any $\varphi \in C^2(\Omega)$ and any $x_0 \in \Omega$
such that $u - \varphi$ has a local maximum at $x_0$
it holds that
\begin{align}\label{eq:viscos1}
F(x_0,u(x_0),D\varphi(x_0),D^2\varphi(x_0)) \leq 0.
\end{align}
A $LSC$ function $u : \Omega \to \mathbb{R}$ is a viscosity supersolution
if for any $\varphi \in C^2(\Omega)$ and any $x_0 \in \Omega$
such that $u - \varphi$ has a local minimum at $x_0$
it holds that
\begin{align}\label{eq:viscos2}
F(x_0,u(x_0),D\varphi(x_0),D^2\varphi(x_0)) \geq 0.
\end{align}
A continuous function is a viscosity solution if it is both a viscosity sub- and viscosity supersolution.

In the following we sometimes drop the word `viscosity' and simply write subsolution, solution and supersolution.
We say that a subsolution (supersolution) is strict in a domain $\Omega$ if the inequality in \eqref{eq:viscos1} (\eqref{eq:viscos2}) is strict,
and we say that a subsolution, supersolution, or solution is \emph{classical} if it is $C^2(\Omega)$.

To avoid technicalities, we chose to stay with the standard viscosity solutions, as defined above, even though other definitions of “weak solutions” might be more suitable for some equations allowed. 
The choice of viscosity solutions is not necessary for our results; other definitions of `weak solutions' can be considered in our arguments as long as such subsolutions (supersolutions) %are USC (LSC) and 
can be compared to classical strict supersolutions (subsolutions) in the sense of Lemma \ref{le:comp-weak} below. For example, it is not necessary for our results that $F$ is continuous, even though this is often assumed in the viscosity solution framework.
%\end{remark}

We intend to prove that the decay rate of positive viscosity solutions of \eqref{eq:main-nonlinear}, vanishing on a portion of the boundary of a $C^{1,1}$-domain, can be estimated
by classical solutions of second order ODIs.
We will build a lower and an upper estimate. 
Let $\Omega \in \mathbb{R}^n$ be a domain, $r > 0$ and $w \in \partial \Omega$. 
For our lower decay estimate in $B(w,r)\cap \Omega$ we assume 
\begin{align}\label{eq:ass_drift_sub}
F(x,s,p,X) \leq \Phi^+(\check t, s, |p|) + \mathcal{P}^+_{\lambda(\check t),\Lambda(\check t)}(X) \quad \text{with}\quad  \check t = d(x,\partial \Omega)%\\\notag %\quad
%\text{whenever} \quad x\in B(w,6r)\cap \Omega, s \in \mathbb{R}_+, p \in \mathbb{R}^n, X \in \mathbb{S}^n,
\end{align}
whenever $x\in B(w,6r)\cap \Omega, s \in \mathbb{R}_+, p \in \mathbb{R}^n, X \in \mathbb{S}^n$,
and for our upper estimate in $B(w,r)\cap \Omega$ we assume 

\begin{align}\label{eq:ass_drift_super}
F(x,s,p,X) \geq \Phi^-(\bar t, s, |p|) + \mathcal{P}^-_{\lambda(\hat t),\Lambda(\hat t)}(X) \quad \text{with}\quad  \hat t = d(x,w) %\quad \text{whenever} \quad s \in \mathbb{R}_+, x, p \in \mathbb{R}^n, X \in \mathbb{S}^n,
\end{align}
whenever $x\in B(w,6r)\cap \Omega, s \in \mathbb{R}_+, p \in \mathbb{R}^n, X \in \mathbb{S}^n$. 
Here and in the following, 
$\lambda \leq \Lambda : (0,\infty) \to (0,\infty)$ are continuous functions such that ${\lambda}$ is nondecreasing and ${\Lambda}$ is nonincreasing. 
$\Phi^-,\Phi^+ : (0,\infty) \times (0,\infty) \times [0,\infty) \to (-\infty,\infty)$ are continuous and such that $\Phi^+$ is 
%$|\Phi^-|, |\Phi^+|$ are
nonincreasing in its first argument, while 
$\Phi^-$ is either nonincreasing in its first argument with $\bar t = d(x,\partial \Omega)$ or nondecreasing with $\bar t = d(x,w)$.
In either case, 
it is assumed that each of \(\Phi^+\) and \(\Phi^-\) has a fixed sign everywhere, but their signs may differ.
Observe that these assumptions allow the ellipticity to vanish
and the lower-order terms to blow up
near the boundary.

Moreover, 
with $\mathcal A_{\lambda,\Lambda}:=\{A \in \mathbb{S}^n : \lambda I \leq A \leq \Lambda I \}$  the Pucci operators are, as usual,
\begin{align*}
\mathcal P_{\lambda,\Lambda}^+(X) =& \sup_{A\in \mathcal A_{\lambda,\Lambda}} -\text{Tr}(AX) = -\lambda \text{Tr}(X^+) + \Lambda \text{Tr}(X^-),\\
\mathcal P_{\lambda,\Lambda}^-(X) =& \inf_{A\in \mathcal A_{\lambda,\Lambda}} -\text{Tr}(AX) = -\Lambda \text{Tr}(X^+) + \lambda \text{Tr}(X^-),
\end{align*}
where $X = X^+ - X^-$ with $X^+ \geq 0$, $X^- \geq 0$, and $X^+  X^- = 0$.

Before stating our main theorems, we recall the following standard definition of sphere conditions:
A point $w \in \partial \Omega$, where $\Omega\subset \mathbb{R}^n$ is a domain,
satisfies the \emph{interior sphere condition} with radius $r_i > 0$ if
there exists $\eta^i \in \Omega$ such that $B(\eta^i, r_i) \subset \Omega$ and
$\partial B(\eta^i, r_i) \cap \partial \Omega = \{w\}$.
Similarly, $w\in \partial \Omega$ satisfies the \emph{exterior sphere condition}
with radius $r_e > 0$ if there exists $\eta^e \in \mathbb{R}^n \setminus \Omega$
such that $B(\eta^e, r_e) \subset \mathbb{R}^n \setminus \Omega$ and $\partial B(\eta^e, r_e) \cap \partial \Omega = \{w\}$.
A point $w \in \partial \Omega$ satisfies the \emph{sphere condition} with radius $r_b$ if it satisfies both the interior and exterior sphere condition with radius $r_b$.
A domain $\Omega\subset \mathbb{R}^n$ is said to satisfy the (interior/exterior) \emph{sphere condition} if the corresponding condition holds for all $w\in \partial \Omega$.
It is well known that $\Omega\subset \mathbb{R}^n$ satisfies the sphere condition if and only if $\Omega$ is a $C^{1,1}$-domain,
see e.g. \cite[Lemma 2.2]{aikawa} for a proof.

%%%%%%%%%%%%%%%%%%%%%%%%%%%%%%%%%%%%%%%%%%%%

We first consider proving lower decay estimates of viscosity supersolutions of \eqref{eq:main-nonlinear}, 
and to do so we will make use of classical solutions of the ODI
\begin{align}\label{eq:ODE}
%\frac{d^2h}{dt^2}
h''(t) \geq \frac{\Phi^+\left(t,h\left(t\right),h'\left(t\right)\right)}{L(t)} +  \frac{n}{r} \frac{\Lambda(t)}{\lambda(t)} h'(t),  \quad t \in(0,r), %\quad h(0)= 0,\quad h(r) = m > 0
\end{align}
where $L(t) = \lambda(t)$ if $\Phi^+$ is nonnegative, and $L(t) = \Lambda(t)$ if $\Phi^+$ is nonpositive. 
We prove the following lower bound.

\begin{theorem}[Lower estimate]\label{th:lower}
Let $\Omega \subset \mathbb{R}^n$ be a domain satisfying the interior sphere
condition with radius $r_i$, $w \in \partial \Omega$ and $r$ s.t. $0<2r < r_i$.
Assume \eqref{eq:ass_drift_sub} and that $v \in LSC(\overline{\Omega} \cap B(w, 6r))$ is a positive viscosity supersolution of \eqref{eq:main-nonlinear} in $\Omega \cap B(w,6r)$ satisfying $v \geq 0$ on $\partial \Omega \cap B(w,6r)$.

If $\check h=\check h(t)$ is an increasing classical solution of \eqref{eq:ODE}, satisfying $\check h(0) = 0$ and $\check h(r) \leq \inf_{\Gamma_{w,r}}
v$, then 
\begin{align*}
v(x) \geq  \check h(d(x, \partial\Omega)) \quad \text{whenever} \quad x \in \Omega \cap B(w, r),
\end{align*}
where $\Gamma_{w,r} = \{x \in \Omega\, |\, r \leq d(x,\partial \Omega) \leq 3r \} \cap B(w,6r)$.
\end{theorem}

%%%%%%%%%%%%%%%%%%%%%%%%%%%%%%%%%%%%%%%%%%%%

We next consider proving upper decay estimates of viscosity subsolutions of \eqref{eq:main-nonlinear}, 
and to do so, we consider the ODI
\begin{align}\label{eq:ODE2}
%\frac{d^2h}{dt^2}
h''(t) \leq \frac{\Phi^-\left(t,h\left(t\right),h'\left(t\right)\right)}{L(t)} -  \frac{n}{r} \frac{\Lambda(t)}{\lambda(t)} h'(t),  \quad t \in(0,r), %\quad h(0)= 0,\quad h(r) = m > 0
\end{align}
where $L(t) = \Lambda(t)$ if $\Phi^-$ is nonnegative, and $L(t) = \lambda(t)$ if $\Phi^-$ is nonpositive. 
We prove: 

\begin{theorem}[Upper estimate]\label{th:upper}
Let $\Omega \subset \mathbb{R}^n$ be a domain satisfying the exterior sphere
condition with radius $r_e$, $w \in \partial \Omega$ and $r$ s.t. $0 < 2r < r_e$.
Assume \eqref{eq:ass_drift_super} and that $u\in USC(\overline{\Omega} \cap B(w, 6r))$ is a viscosity subsolution of \eqref{eq:main-nonlinear} in
$\Omega \cap B(w, 6r)$ satisfying $u \leq 0$ on $\partial \Omega \cap B(w,6r)$.

If $\hat h=\hat h(t)$ is an increasing classical solution of \eqref{eq:ODE2}, satisfying $\hat h(0) \geq 0$ and $\hat h(r) \geq \sup_{B(w,6 r)\cap \Omega} u$, then 
\begin{align*}
 u(x) \leq \, \hat h(d(x, \partial\Omega)) \quad \text{for} \quad x \in \Omega \cap B(w, r).
\end{align*}
\end{theorem}

%%%%%%%%%%%%%%%%%%%%%%%%%%%%%%%%%%%%%%%%%%%%

Theorem \ref{th:lower} and Theorem \ref{th:upper} imply that the boundary Harnack inequality for viscosity solutions of \eqref{eq:main-nonlinear} can be proved simply by proving it for the solutions of the ODIs \eqref{eq:ODE} and \eqref{eq:ODE2}. 
In particular:

\begin{corollary}\label{cor:BHI}
Suppose that $\Omega \in \mathbb{R}^n$ is a domain satisfying the sphere condition with radius $r_0$, $w \in \partial \Omega$, and $0 < 2r <r_0$.  
Assume \eqref{eq:ass_drift_sub} and \eqref{eq:ass_drift_super} and let $u_1, u_2 \in C(\overline \Omega \cap B(w,6r))$ be positive viscosity solutions of \eqref{eq:main-nonlinear} in $\Omega\cap B(w, 6r)$, satisfying $u_1 = 0 = u_2$ on $\partial \Omega \cap B(w, 6r)$. 

If $\check h=\check h(t)$ is an increasing classical solution of \eqref{eq:ODE}, $\check h(0) = 0$ and $\check h(r) \leq \inf_{\Gamma_{w,r}} u_i,$ 
and if $\hat h=\hat h(t)$ is an increasing classical solution of \eqref{eq:ODE2}, $\hat h(0) = 0$ and $\hat h(r) \geq \sup_{B(w,6 r)\cap \Omega} u_i$, $i = 1,2$, 
then  
\begin{align*}
\check  h(d(x, \partial\Omega)) \, \leq u_i(x) \leq \, \hat h(d(x, \partial\Omega)) \quad \text{for} \quad x \in \Omega \cap B(w, r), \quad i = 1,2.
\end{align*}
Moreover, the boundary Harnack inequality for $u_1$ and $u_2$ is implied by the boundary Harnack inequality for $\check h$ and $\hat h$, i.e.: If there exists a constant $A$ such that
\begin{align*}
1 \leq \frac{\hat h(t)}{\check h(t)} \leq A \quad \text{whenever} \quad t \in (0,r),
\end{align*}
then
\begin{align*}
A^{-1} \leq \frac{u_1(x)}{u_2(x)} \leq A \quad \text{whenever} \quad  x \in \Omega \cap B(w, r).
\end{align*}
\end{corollary}
\begin{proof}
This follows directly from Theorems \ref{th:lower} and \ref{th:upper}.
\end{proof}

\bigskip

\noindent
We note that, by construction, the inequality 
$\check h \leq \hat h$ always holds, provided that 
$\check h(0) \leq \hat h(0)$ and $\check h(r) \leq \hat h(r)$; this reflects a comparison principle for solutions of the `help' ODIs.

It may be disappointing to the reader that we do not proceed by establishing `as weak as possible' conditions on the ODIs \eqref{eq:ODE} and \eqref{eq:ODE2} that would ensure, for example, bounds of the form 
$t \lesssim \check h(t) \leq \hat h \lesssim t$. 
However, such conditions often turn out to be technical and non-trivial to verify. 
Therefore, we here instead refer the reader to investigate the particular ODIs corresponding to the PDE under consideration. 
In a less general setting conditions of this type, which are sharp in certain senses, can be found in \cite{LOT20}. 
We will demonstrate below that, in several situations, our theorems allow us to obtain meaningful results even when these conditions in \cite{LOT20} are not satisfied.

%%%%%%%%%%%%%%%%%%%%%%%%%%%%%%%%%%%%%%%%%%%%
%%%%%%%%%%%%%%%%%%%%%%%%%%%%%%%%%%%%%%%%%%%%
%%%%%%%%%%%%%%%%%%%%%%%%%%%%%%%%%%%%%%%%%%%%

\subsection{Proofs of main results: Theorems \ref{th:lower} and \ref{th:upper}}

The proofs of theorems \ref{th:lower} and \ref{th:upper} rely on the following two lemmas, 
giving a classical sub- and supersolution of the extremal PDEs. 
\begin{lemma}
\label{le:barrier_s}
Suppose that $\Omega \subset \mathbb{R}^n$, 
$r > 0$ and assume that $\check h$ is an increasing classical solution of the ODI \eqref{eq:ODE} for $ t \in (0,r)$. 
If $B(y,2r) \subset \Omega$, then 
\begin{align*}%\label{eq:super}
U(x) =   \check h(2r - |x-y|) \quad \text{satisfies} \quad \Phi^+(d(x,\partial \Omega), U, |DU|) + \mathcal{P}^+_{\lambda,\Lambda}(D^2U) < 0,
\end{align*}
whenever $x \in B(y, 2r) \setminus \overline{B(y, r)}$.
\end{lemma}

\begin{lemma}
\label{le:barrier_s2}
Suppose that $\Omega \subset \mathbb{R}^n$, 
$w \in \partial \Omega$, 
$r > 0$ and assume that $\hat h$ is an increasing classical solution of the ODI \eqref{eq:ODE2} for $t \in (0,r)$. 
If $B(y,r) \subset \complement \Omega$, then 
$$
V(x) = \hat h(|x-y|-r) \quad  \text{satisfies}\quad \Phi^-(\bar t, V, |DV|) + \mathcal{P}^-_{\lambda,\Lambda}(D^2V) > 0,
$$
whenever $x \in \Omega \cap (B(y, 2r) \setminus \overline{B(y, r)}) \cap \{x: d(x,w) \geq |x-y|-r\}$.
\end{lemma}

%%%%%%%%%%%%%%%%%%%%%%%%%%%%%%%%%%%%%%%%%%%%
%%%%%%%%%%%%%%%%%%%%%%%%%%%%%%%%%%%%%%%%%%%%
%%%%%%%%%%%%%%%%%%%%%%%%%%%%%%%%%%%%%%%%%%%%

\noindent
{\bf Proof of Lemma \ref{le:barrier_s}.}
Put $\Xi = 2r - |x-y|$.
Differentiating yields
\begin{align}\label{eq:grad-beg}
\frac{\partial U}{\partial x_i} = -\frac{x_i-y_i}{|x - y|} \check h'\left(\Xi\right), \quad 1 \leq i \leq n, \quad
\vert D U \vert = \check h'\left(\Xi\right),
\end{align}
and the second derivatives become
\begin{align*}
\frac{\partial^2 U}{\partial x_i^2}
= %&
 \left( \frac{x_i-y_i}{|x - y|}\right)^2 \check  h''(\Xi)%\\
+%&
  \left(-\frac{1}{|x-y|} + \frac{(x_i-y_i)^2}{|x-y|^3} \right) \check h'\left(\Xi\right),
\end{align*}
for $1 \leq i \leq n$, 
giving
\begin{align*}
\text{Tr}(D^2 U) %&
=  \check h''(\Xi) - \frac{n-1}{|x-y|} \check h'\left(\Xi\right).
\end{align*}
Assume $\Phi^+\geq 0$, 
then from the ODI \eqref{eq:ODE} we have 
$$
\check h''(t) \geq \frac{\Phi^+\left(t,\check h\left(t\right),\check h'\left(t\right)\right)}{\lambda(t)} +  \frac{n}{r} \frac{\Lambda(t)}{\lambda(t)} \check h'(t)
$$
and since $U$ is radial and $\check h'>0$, 
we can decompose $D^2 U = (D^2 U^+) - (D^2 U^-)$ so that
\begin{align*}
\text{Tr} (D^2 U^+) &\geq \frac{\Phi^+\left(\Xi,\check h\left(\Xi\right),\check h'\left(\Xi\right)\right)}{\lambda(\Xi)}  +  \frac{n}{r} \frac{\Lambda(\Xi)}{\lambda(\Xi)} \check h'(\Xi),\\ %\quad \textrm{and} \quad
\text{Tr} (D^2 U^-) &\leq  \frac{n-1}{|x-y|}\check h'\left(\Xi\right).
\end{align*}
The facts that $U = \check h\left(\Xi\right)$, and $|DU| = \check h'(\Xi)$ according to \eqref{eq:grad-beg}, give
\begin{align}\label{eq:super_proof_2}
\Phi^+&(d(x,\partial \Omega), U, |DU|) + \mathcal{P}^+_{\lambda(d(x,\partial\Omega)),\Lambda(d(x,\partial\Omega))}(D^2U) \notag\\
&\leq \Phi^+\left(d(x,\partial\Omega), \check h\left(\Xi\right), \check h'\left(\Xi\right)\right)
- \lambda(d(x,\partial\Omega)) \frac{\Phi^+\left(\Xi,\check h\left(\Xi\right),\check h'\left(\Xi\right)\right)}{\lambda(\Xi)}\notag\\ 
&+ \Lambda(d(x,\partial\Omega)) \frac{n-1}{|x-y|}\check h'\left(\Xi\right)
-\lambda(d(x,\partial\Omega))  \frac{n}{r} \frac{\Lambda(\Xi)}{\lambda(\Xi)} \check h'(\Xi).
\end{align}
By the assumptions in the lemma, geometry implies
$\Xi = 2r - |x-y| \leq d(x,\partial \Omega)$, 
so that by the monotonicity assumptions on $\lambda, \Lambda$ and $\Phi(t, \cdot,\cdot)$, 
\begin{align*}
\lambda(d(x,\partial\Omega)) \geq \lambda(\Xi), \quad
\Lambda(d(x,\partial\Omega)) \leq \Lambda(\Xi), \quad 
\Phi^+(d(x,\partial\Omega),\cdot,\cdot) \leq \Phi^+(\Xi,\cdot,\cdot).
\end{align*}
This, together with \eqref{eq:super_proof_2}, allows us to conclude the following:
\begin{align}\label{eq:super_proof_2_2}
%F(x,U,DU,D^2U) 
\Phi^+(d(x,\partial \Omega), U, |DU|) +& \mathcal{P}^+_{\lambda(d(x,\partial\Omega)),\Lambda(d(x,\partial\Omega))}(D^2U) \notag\\
&\leq  \Lambda(\Xi) \check h'\left(\Xi\right)  \left(\frac{n-1}{|x-y|} - \frac{n}{r} \right) < 0,
\end{align}
where the last inequality follows from $r < |x-y| < 2r$ and $\check h' > 0$. 
This concludes the proof in the case $\Phi^+ \geq 0$.

If $\Phi^+ \leq 0$, then by \eqref{eq:ODE} we have
$
\check h''(t) \geq 
\frac{\Phi^+\left(t,\check h\left(t\right),\check h'\left(t\right)\right)}{\Lambda(t)} +  \frac{n}{r} \frac{\Lambda(t)}{\lambda(t)} \check h'(t)
$
so that
we obtain, after decomposing the Hessian, that
\begin{align*}%\label{eq:tracetjoho}
\text{Tr}(D^2U)^+ \geq (A-B)^+,
\quad
\text{Tr}(D^2U)^- \leq (A-B)^- + \frac{n-1}{|x -y|}h'(\Xi),
\end{align*}
in which $t^\pm=\max\{\pm t,0\}$, 
$$
B= -\frac{\Phi^+\left(\Xi,\check h\left(\Xi\right),\check h'\left(\Xi\right)\right)}{\Lambda(\Xi)}\geq 0, \quad \text{and}\quad A=\frac{n}{r} \frac{\Lambda(\Xi)}{\lambda(\Xi)} \check h'(\Xi) > 0.
$$
Using that $\Lambda (A-B)^--\lambda (A-B)^+\leq \Lambda B-\lambda A$ we obtain 
\begin{align*}%\label{eq:super_proof_33}
\Phi^+&(d(x,\partial \Omega), U, |DU|) + \mathcal{P}^+_{\lambda(d(x,\partial\Omega)),\Lambda(d(x,\partial\Omega))}(D^2U) \\
\leq&  \Phi^+\left(d(x,\partial\Omega), \check h\left(\Xi\right), \check h'\left(\Xi\right)\right) +  \Lambda(d(x,\partial\Omega)) \frac{n-1}{|x-y|}\check h'\left(\Xi\right)\\
&+\Lambda(d(x,\partial\Omega)) (A - B)^- - \lambda(d(x,\partial\Omega))(A - B)^+\\
\leq&  \Phi^+\left(d(x,\partial\Omega), \check h\left(\Xi\right), \check h'\left(\Xi\right)\right) +  \Lambda(d(x,\partial\Omega)) \frac{n-1}{|x-y|}\check h'\left(\Xi\right)\\
&+\Lambda(d(x,\partial\Omega)) B - \lambda(d(x,\partial\Omega))A \\
&\leq  \Phi^+\left(d(x,\partial\Omega), \check h\left(\Xi\right), \check h'\left(\Xi\right)\right)
- \Lambda(d(x,\partial\Omega)) \frac{\Phi^+\left(\Xi,\check h\left(\Xi\right),\check h'\left(\Xi\right)\right)}{\Lambda(\Xi)}\\
&+   \Lambda(d(x,\partial\Omega)) \frac{n-1}{|x-y|}\check h'\left(\Xi\right)\notag 
  -\lambda(d(x,\partial\Omega))  \frac{n}{r}\frac{\Lambda(\Xi)}{\lambda(\Xi)} \check h'(\Xi).
\end{align*}
Using the monotonicity of $\lambda$, $\Lambda$ and $\Phi^+(t,\cdot,\cdot)$, and $r < |x-y| < 2r$ and $\check h' > 0$, 
we realize that this inequality also boils down to \eqref{eq:super_proof_2_2}. 
The proof of Lemma \ref{le:barrier_s} is complete. $\hfill\Box$\\

%%%%%%%%%%%%%%%%%%%%%%%%%%%%%%%%%%%%%%%%%%%%
%%%%%%%%%%%%%%%%%%%%%%%%%%%%%%%%%%%%%%%%%%%%
%%%%%%%%%%%%%%%%%%%%%%%%%%%%%%%%%%%%%%%%%%%%

\noindent
{\bf Proof of Lemma \ref{le:barrier_s2}.} 
The arguments are very similar to the proof of Lemma \ref{le:barrier_s}. %we just briefly outline the proof. 
With $\zeta = |x-y| - r$ we have
\begin{align*}
\text{Tr}(D^2 V) 
= \hat h''(\zeta) + \frac{n-1}{|x-y|}\hat h'\left(\zeta\right),
\end{align*}
and when $\Phi^-\leq 0$ it follows from the ODI in \eqref{eq:ODE2} that 
\begin{align*}
\text{Tr} (D^2 V^+) &\leq \frac{n-1}{|x-y|} \hat h'\left(\zeta\right)\\ %\quad \textrm{and} \quad
\text{Tr} (D^2 V^-) &\geq  -\frac{\Phi^-\left(\zeta,\hat h\left(\zeta\right),\hat h'\left(\zeta\right)\right)}{\lambda(\zeta)}  +  \frac{n}{r} \frac{\Lambda(\zeta)}{\lambda(\zeta)} \hat h'(\zeta),  
\end{align*}
so that
\begin{align*}
\Phi^-&(\bar t, V, |DV|) + \mathcal{P}^-_{\lambda(\hat t),\Lambda(\hat t)}(D^2V) \\
&\geq \Phi^-(\bar t, \hat h\left(\zeta\right), \hat h'\left(\zeta\right)) 
- \lambda(\hat t) \frac{\Phi^-\left(\zeta,\hat h\left(\zeta\right),\hat h'\left(\zeta\right)\right)}{\lambda(\zeta)} \\
&- \Lambda(\hat t) \frac{n-1}{|x-y|}\hat h'\left(\zeta\right)
+\lambda(\hat t)  \frac{n}{r} \frac{\Lambda(\zeta)}{\lambda(\zeta)} \hat h'(\zeta).
\end{align*}
Geometry and the assumption in the Lemma imply $d(x,\partial \Omega) \leq \zeta \leq d(x,w)$, 
which, together with the monotonicity assumption on $\Phi^-$; 
($\Phi^-$ is either nonincreasing with $\bar t = d(x,\partial \Omega)$ or nondecreasing with $\bar t = d(x,w)$),
implies
\begin{align}\label{eq:monoton_phi-}
\Phi^-(\bar t, \cdot, \cdot) \geq \Phi^-(\zeta,\cdot,\cdot).
\end{align}
Moreover, $\lambda$ is nondecreasing 
so that $\lambda(\hat t) \geq \lambda(\zeta)$.
Therefore (recall $\Phi^- \leq 0$), 
$$
\Phi^-(\bar t, \hat h\left(\zeta\right), \hat h'\left(\zeta\right)) \geq \frac{\lambda(\hat t)}{\lambda(\zeta)}\Phi^-(\zeta,\hat h\left(\zeta\right),\hat h'\left(\zeta\right)),
$$
and thus, it is enough to prove
\begin{align}\label{eq:pedagogiskreferens}
- \Lambda(\hat t) \frac{n-1}{|x-y|}\hat h'\left(\zeta\right)
 +\lambda(\hat t)  \frac{n}{r} \frac{\Lambda(\zeta)}{\lambda(\zeta)} \hat h'(\zeta) > 0.
\end{align}
Once again $\lambda(\hat t) \geq \lambda(\zeta)$,  $\Lambda(\hat t) \leq \Lambda(\zeta)$, 
and thus we only need to show the following;
\begin{align*}
- \Lambda(\zeta) \frac{n-1}{|x-y|}\hat h'\left(\zeta\right)
 + \frac{n}{r} {\Lambda(\zeta)} \hat h'(\zeta) > 0,
\end{align*}
which holds by $r < |x-y| < 2r$ and $\hat h' > 0$. This concludes the proof in the case $\Phi^-\leq 0$.  

If $\Phi^-\geq 0$ we have 
\begin{align*}
\text{Tr} (D^2 V^+) &\leq \frac{n-1}{|x-y|}\hat h'\left(\zeta\right) +  (A-B)^+ \\
%\quad \textrm{and} \quad
\text{Tr}(D^2 V^-) &\geq (A-B)^-,  
\end{align*}
in which $t^\pm=\max\{\pm t,0\}$, 
$$
A=\frac{\Phi^-\left(\zeta,\hat h\left(\zeta\right),\hat h'\left(\zeta\right)\right)}{\Lambda(\zeta)}\geq 0, \quad \text{and}\quad B=\frac{n}{r} \frac{\Lambda(\zeta)}{\lambda(\zeta)} \hat h'(\zeta) > 0.
$$
Moreover, $\Lambda (A-B)^+-\lambda (A-B)^-\leq \Lambda A-\lambda B$, so that
\begin{align*}
\Phi^-&(\bar t, V, |DV|) + \mathcal{P}^-_{\lambda(\hat t),\Lambda(\hat t)}(D^2V) \\
&\geq \Phi^-(\bar t, \hat h\left(\zeta\right), \hat h'\left(\zeta\right)) - \Lambda(\hat t) \frac{n-1}{|x-y|}\hat h'\left(\zeta\right)
- \Lambda(\hat t) (A - B)^+ + \lambda(\hat t) (A - B)^-\\
&\geq \Phi^-(\bar t, \hat h\left(\zeta\right), \hat h'\left(\zeta\right)) 
- \Lambda(\hat t) \frac{\Phi^-\left(\zeta, \hat h\left(\zeta\right), \hat h'\left(\zeta\right)\right)}{\Lambda(\zeta)} \\
&- \Lambda(\hat t) \frac{n-1}{|x-y|}\hat h'\left(\zeta\right)
+\lambda(\hat t)  \frac{n}{r} \frac{\Lambda(\zeta)}{\lambda(\zeta)} \hat h'(\zeta).
\end{align*}
By the assumptions, we conclude that \eqref{eq:monoton_phi-} holds also in this case, as well as $\Lambda(\hat t) \leq \Lambda(\zeta)$, so that  
$$
\Phi^-(\bar t, \hat h\left(\zeta\right), \hat h'\left(\zeta\right)) \geq \frac{\Lambda(\hat t)}{\Lambda(\zeta)}\Phi^-\left(\zeta, \hat h\left(\zeta\right), \hat h'\left(\zeta\right)\right),
$$
and it is enough to prove
\begin{align*}
%\Phi^-(d(x,w), V, |DV|) + \mathcal{P}^-_{\lambda,\Lambda}(D^2V) \geq 
- \Lambda(\hat t) \frac{n-1}{|x-y|}\hat h'\left(\zeta\right)
 +\lambda(\hat t)  \frac{n}{r} \frac{\Lambda(\zeta)}{\lambda(\zeta)} \hat h'(\zeta) > 0,
\end{align*}
which is equivalent to \eqref{eq:pedagogiskreferens}. 
We are thus back in the same situation as with 
$\phi^- \leq 0$ and 
the proof of Lemma \ref{le:barrier_s2} is complete.
$\hfill \Box$ \\

%%%%%%%%%%%%%%%%%%%%%%%%%%%%%%%%%%%%%%%%%%%%
%%%%%%%%%%%%%%%%%%%%%%%%%%%%%%%%%%%%%%%%%%%%
%%%%%%%%%%%%%%%%%%%%%%%%%%%%%%%%%%%%%%%%%%%%

To make use of the sub- and supersolutions constructed in Lemmas \ref{le:barrier_s} and \ref{le:barrier_s2} we will rely on the following simple lemma, 
which is immediate in the setting of viscosity solutions.

\begin{lemma}\label{le:comp-weak}
Assume \eqref{eq:deg-ellipt-proper} and let $\Omega$ be a bounded domain, $u \in USC(\overline{\Omega})$ a viscosity subsolution, and $v \in LSC(\overline{\Omega})$ a viscosity supersolution of \eqref{eq:main-nonlinear}, satisfying $u \leq v$ on $\partial \Omega$.
If either $u$ is a classical strict subsolution, or $v$ is a classical strict supersolution,
then $u < v$ in $\Omega$.
\end{lemma}

\noindent
{\bf Proof.}
Assume that $v$ is a classical strict supersolution and also, by way of contradiction, that $u \geq v$ somewhere in $\Omega$. 
By $USC$ the function $u-v$ then attains a maximum $\geq 0$ at some point $x_0 \in \Omega$.
Since $v \in C^2(\Omega)$, $u - v$ has a maximum at $x_0$, and $u$ is a viscosity subsolution, it follows by definition that
\begin{align}\label{eq:johej-lemma1}
F(x_0, u(x_0), Dv(x_0), D^2 v(x_0)) \leq 0.
\end{align}
But since $v$ is a classical strict supersolution, 
we have 
$$
F(x,v(x), Dv(x), D^2v(x)) > 0 \quad \text{whenever} \quad x \in \Omega,
$$
and as $u(x_0) \geq v(x_0)$ it follows from \eqref{eq:deg-ellipt-proper} that
$$
F(x_0, u(x_0), Dv(x_0), D^2 v(x_0)) \geq F(x_0, v(x_0), Dv(x_0), D^2 v(x_0)) > 0.
$$
This contradicts \eqref{eq:johej-lemma1} and hence we have proved the lemma in case $v$ is a classical strict supersolution. 
The other case is similar. $\hfill \Box$\\

We now give the proof of our lower estimate, i.e. Theorem \ref{th:lower}. 
Armed with Lemmas \ref{le:barrier_s} and \ref{le:comp-weak}, 
we can roughly follow the lines of \cite[Lemma 3.1]{aikawa} to obtain the Theorem.\\

\noindent
{\bf Proof of Theorem \ref{th:lower}.}
Take $x \in \Omega \cap B(w, r)$ and let $\eta \in \partial \Omega$ be such that $d(x, \partial\Omega) = |x - \eta|$.
By the interior sphere condition at $\eta$ we can find a point
$\eta^i$ such that $B(\eta^i, r_i) \subset \Omega$ and $\eta \in \partial B(\eta^i, r_i)$.
Now, choose the point $\eta^i_{2r}$ such that $\eta= \eta^i_{2r} +2r \frac{\eta-\eta^i}{|\eta-\eta^i|}$, i.e. $\eta^i_{2r} \in \Omega$ lies on a straight line $\ell$ between $\eta$ and $\eta^i$ at a distance $2r$ from the boundary $\partial \Omega$. Then ${B(\eta^i_{2r}, 2 r)}\subset \Omega$ so ${v}$ is a positive viscosity supersolution in $B(\eta^i_{2r}$, 2r).  By construction $|x-\eta| < r$ so $|\eta^i_{2r}-x| > r$ and by the interior sphere condition we get that $x$ lies on the line $\ell$ and thus $x \in A :=B(\eta^i_{2r},2r)\setminus \overline{	B(\eta^i_{2r}, r)}$.

Let $\check h$ be a solution of \eqref{eq:ODE} according to the theorem, put $U = \check h(2r - |x-\eta^i_{2r}|)$ and note that then, by \eqref{eq:ass_drift_sub}, Lemma \ref{le:barrier_s} and the boundary values of $\check h$, $U$ is a strict classical subsolution of \eqref{eq:main-nonlinear} in $A$ satisfying
\begin{align*}
U \leq \inf_{\Gamma_{w,r}} v \quad\text{on}\quad \partial B(\eta^i_{2r}, r) \quad \text{and} \quad  U \leq 0 \quad\text{on}\quad\partial B(\eta^i_{2r}, 2r). 
\end{align*}
Now ${B(\eta^i_{2r}, r)} \subset \Gamma_{w,r}$ and hence $U \leq v$ on the boundaries of the annulus $A$ which is contained in $\Omega\cap B(w,6r)$ where $v$ is a viscosity supersolution, so that the weak comparison principle in Lemma \ref{le:comp-weak} implies $U \leq v$ in $A$.
Summarizing the above construction, 
we can conclude that
\begin{align*}
v(x) \geq U(x) = \check h(2r - |x-\eta^i_{2r}|) = \check h(d(x,\partial \Omega))\quad\text{whenever}\quad x\in\Omega\cap B(w,r),
\end{align*}
and the proof is complete. 
$\hfill\Box$\\

%%%%%%%%%%%%%%%%%%%%%%%%%%%%%%%%%%%%%%%%%%%%
%%%%%%%%%%%%%%%%%%%%%%%%%%%%%%%%%%%%%%%%%%%%
%%%%%%%%%%%%%%%%%%%%%%%%%%%%%%%%%%%%%%%%%%%%

While the lower estimate (Theorem \ref{th:lower}) followed easily--once the classical subsolution was constructed in Lemma \ref{le:barrier_s})--by applying the argument in \cite[Lemma 3.1]{aikawa}, the upper estimate (Theorem \ref{th:upper}) requires additional constructions. This is due to the potentially vanishing ellipticity and the fact that, in the upper estimate, the barrier intersects the boundary(see Figure~\ref{figgen}), unlike in the lower estimate. 
This is the reason why we allow ellipticity to blow up only at a point in the upper estimate, whereas in the lower estimate, it may blow up along the entire boundary.\\

%%%%%%%%%%%%%%%%%%%%%%%%%%%%%%%%%%%%%%%%%%%figgen%%%%%%%%

%\iffalse
%Här ligger koden som inte ska kompileras.
%Den kan omfatta flera rader.

\vspace*{1cm}
\begin{center}
\begin{tikzpicture}[scale=2.5]
\draw[thick] (0,2) .. controls (1,2.2) and (2,1.9) .. (3,2.1) coordinate[pos=0.85] (x0) .. controls (4,2.2) and (5,1.9) .. (6,2) coordinate[pos=0.02] (x1); 
\coordinate (x2) at ($(x1)+(0,1)$); 
%

%
%\draw[red,dashed] (x0) -- (x2);
%
\coordinate (M) at ($(x0)!0.5!(x2)$); \path let \p1=(x0), \p2=(x2) in \pgfextra{ \pgfmathsetmacro{\c}{veclen(\x2-\x1,\y2-\y1)/2} \pgfmathsetmacro{\a}{0.5} \pgfmathsetmacro{\b}{sqrt(\c*\c-\a*\a)} \pgfmathsetmacro{\ang}{atan2(\y2-\y1,\x2-\x1)} };

\node at ($(x2)+(0,1.5)$) {$A$};

%\node[font=\large] at ($(x0)+(-2.3,-0.2)$){$\Omega$};
%\node[font=\large] at ($(x0)+(-2.8,-0.05)$){$\partial\Omega$};
\node[below] at ($(x1)+(-2.8,-0.7)$) {$\Omega$}; 
\node[above] at ($(x1)+(-2.8,-0.05)$) {$\partial\Omega$};
\node[above] at ($(x1)+(-2.8,2.8)$) {$\mathbb{R}^n\setminus\Omega$};

%\node[above, font=\scriptsize] at ($(x1)+(-1.1,0.2)$) {$B(w,r)$};
%

\begin{scope} \clip (0,0) -- (0,2) .. controls (1,2.2) and (2,1.9) .. (3,2.1) .. controls (4,2.2) and (5,1.9) .. (6,2) -- (6,0) -- cycle; \fill[gray!30,even odd rule] (x2) circle (2) (x2) circle (1); \end{scope}

\fill (x1) circle (1pt) node[below=1.8pt,font=\scriptsize] {$\eta$}; \draw (x0) circle (1); \draw (x2) circle (1); 
\fill (x2) circle (1pt) node[right=1.8pt,font=\scriptsize] {$\eta_r^e$}; \draw (x2) circle (2); 
\coordinate (x3) at ($(x1)+(0,-0.4cm)$); \fill (x3) circle (1pt) node[below=1.8pt,font=\scriptsize] {$x$};

\path[name path=big] (x2) circle (1);

\coordinate (C) at ($(x0)+(0,1)$); 
\draw[dashed,gray] (C) circle (1); 
\fill (C) circle (1pt);
\draw[dashed,gray] (C) circle (2);
\node[above=1.8pt, font=\scriptsize] at (C) {$\eta_r^w$};
%\node at ($(C)+(0,1.5)$) {$A^w$};
\node at ($(C)-(1.75,0)$) {$A^w$};

\coordinate (M) at ($(x0)!0.5!(x2)$); 

\begin{scope} \clip (0,0) -- (0,2) .. controls (1,2.2) and (2,1.9) .. (3,2.1) .. controls (4,2.2) and (5,1.9) .. (6,2) -- (6,0) -- cycle; \clip (x2) circle (2); \begin{scope}[shift={(M)},rotate=64.41537]

\draw[thick,name path global=leftbranch] plot[samples=2000,domain=-88:88] ({-0.5/cos(\x)},{0.313*tan(\x)}); 
\draw[thick] plot[samples=2000,domain=0.001:88] ({-0.5/cos(\x)},{0.313*tan(\x)}); \draw[thick] plot[samples=2000,domain=-88:-0.001] ({0.5/cos(\x)},{0.313*tan(\x)}); \draw[thick] plot[samples=2000,domain=0.001:88] ({0.5/cos(\x)},{0.313*tan(\x)}); 

\end{scope} \end{scope}
\path[name path=outercircle] (x2) circle (2);
%
%\path[name path global=ringarc] (x2) circle (2);
%
%\path[name intersections={of=testbranch and outercircle, by={C,D}}]; \fill[blue] (C) circle (1pt); \fill[magenta] (D) circle (1pt);
%
\path[name intersections={ of=leftbranch and outercircle, by={A,B} }];
%nedan fixade hyperbeln vit
\begin{scope} \clip (0,0) -- (0,2) .. controls (1,2.2) and (2,1.9) .. (3,2.1) .. controls (4,2.2) and (5,1.9) .. (6,2) -- (6,0) -- cycle; \clip (x2) circle (2); \begin{scope}[shift={(M)},rotate=64.41537] 
\path[fill=white] plot[samples=2000,domain=-88:88] ({-0.5/cos(\x)},{0.313*tan(\x)}) -- (A) arc[start angle=-90,end angle=-225,radius=2] -- cycle; \draw[thick] plot[samples=2000,domain=-88:88] ({-0.5/cos(\x)},{0.313*tan(\x)});

\end{scope} \end{scope}

\draw (x2) circle (2);

\path[name path=outercircle] (x2) circle (2); 
\node at (1.93,1.79) {$\Gamma$};
\node at (2.73,1.58) {$\Gamma$};
\fill (x0) circle (1pt) node[below=1.8pt,font=\scriptsize] {$w$}; 

\draw[thick] (0,2) .. controls (1,2.2) and (2,1.9) .. (3,2.1) .. controls (4,2.2) and (5,1.9) .. (6,2); 
\end{tikzpicture}
\vspace{-1.5cm}
\captionof{figure}{Geometry of the proof of Theorem~\ref{th:upper}. The shaded region represents the set $A \cap \{x: d(x,w) \geq |x-\eta^e_r| - r\} \cap \Omega$, where the comparison principle is applied.} \label{figgen}
\end{center}

%Här avslutas avstängningen
%\fi

%%%%%%%%%%%%%%%%%%%%%%%%%%%%%%%%%%%%%%%%%%%%
\noindent
{\bf Proof of Theorem \ref{th:upper}.}
Take $x \in \Omega \cap B(w, r)$ and let $\eta \in \partial \Omega$ be such that $d(x, \partial\Omega) = |x - \eta|$.
By the exterior sphere condition at $\eta$ we can find a point $\eta^e$ such that $B(\eta^e, r_e) \subset\mathbb{R}^n \setminus\Omega$
and $\eta \in \partial B(\eta^e, r_e)$.
Now, choose the point $\eta^e_{r}$ which is such that $\eta= \eta^e_{r} + r \frac{\eta-\eta^e}{|(\eta-\eta^e)|}$ (i.e., $\eta^e_r \not \in \Omega$ lies on the straight line $\gamma$ between $\eta$ and $\eta^e$ at a distance $r$ from the boundary $\partial \Omega$)

Let $\hat h$ be a solution of \eqref{eq:ODE2} according to the theorem, put $V = \hat h(|x-\eta^e_{r}|-r)$ and note that then, by \eqref{eq:ass_drift_super}, Lemma \ref{le:barrier_s2} and the boundary values of $\hat h$, $V$ is a strict classical supersolution of \eqref{eq:main-nonlinear} in the domain 
$$
A \cap \{x: d(x,w) \geq |x-\eta^e_r| - r\}, \quad \text{where} \quad A = B(\eta^e_r,2r)\setminus \overline{B(\eta^e_r,r)}.
$$
Moreover, $V$ satisfies
\

\begin{align*}
V \geq 0 \quad\text{on}\quad \partial B(\eta^e_{r}, r) \quad \text{and} \quad  V \geq \sup_{B(w,6r)\cap \Omega} u \quad\text{on}\quad\partial B(\eta^e_{r}, 2r). 
\end{align*}

We now intend to use the comparison principle in Lemma \ref{le:comp-weak} and $V$ to show that $u(x) \leq V(x)$.
The theorem then follows since $x$ was arbitrary and we have $u(x) \leq V(x) = \hat h(|x-\eta^e_{r}|-r) = \hat h(d(x,\partial \Omega))$.

By construction, $u \leq V$ on $\partial \left( A \cap \Omega \right)$, in which $u$ is a subsolution, but $V$ is a supersolution only in the set $A \cap \{x: d(x,w) > |x-\eta^e_r| - r\}$. 
If 
\begin{align}\label{eq:tjohejja}
V(y) \geq u(y) \quad \text{on}\quad  \Gamma = \left(A \cap \Omega \cap \{x: d(x,w) = |x-\eta^e_r| - r\}\right),
\end{align}
then comparison in $A \cap \{x: d(x,w) \geq |x-\eta^e_r| - r\} \cap \Omega$ gives us the desired result. 
Hence, it remains to show \eqref{eq:tjohejja} and to do so we proceed as follows.
By the exterior ball condition at $w \in \partial \Omega$, there exists $\eta^w_r$ such that $B(\eta^w_r,r) \subset \mathbb{R}^n\setminus \Omega$ and $w \in \partial B(\eta^w_r,r)$. 
Using Lemma \ref{le:barrier_s2} once more we see that 
$V^w = \hat h(|x-\eta^w_{r}|-r)$ is a strict classical supersolution of \eqref{eq:main-nonlinear} in the annulus  $A^w = B(\eta^w_r,2r)\setminus \overline{B(\eta^w_r,r)}$, 
(which follows since $d(z,w) \geq |\eta^w_r-z|-r$ always holds in $A^w$). Clearly
\begin{align*}
V^w \geq 0 \quad\text{on}\quad \partial B(\eta^w_{r}, r) \quad \text{and} \quad  V^w \geq \sup_{B(w,6r)\cap \Omega} u \quad\text{on}\quad\partial B(\eta^w_{r}, 2r). 
\end{align*}
For any $y\in \Gamma$ it holds, by definition of $\Gamma$, that 
$d(y,w) = |\eta^e_r - y| - r$. Thus $d(y,w) + r = |\eta^e_r - y|$, and by the triangle inequality $d(y,\eta^w_r) \leq d(y,w) + r$. 
Therefore 
$$
d(y,\eta^w_r) \leq d(y,\eta^e_r) \quad \text{whenever} \quad  y\in \Gamma,
$$
and since $\hat{h}$ is increasing, it follows by construction of $V$ and $V^w$ that we have 
$V(y) \geq V^w(y)$.
Now, $V^w \geq u$ on $\partial \left(\Omega \cap A^w\right)$ and since $\Gamma\subset\left(\Omega \cap A^w\right)$, it follows by the comparison principle in Lemma~\ref{le:comp-weak} that this holds also on $\Gamma$. 
Therefore $V(y) \geq V^w(y) \geq u(y)$ whenever $y\in\Gamma$, and \eqref{eq:tjohejja} is established. 
The proof is complete. $\hfill\Box$
%%%%%%%%%%%%%%%%%%%%%%%%%%%%%%%%%%%%%%%%%%%%

\subsection{Examples}
\label{sec:examples}

In this section we apply Theorem \ref{th:lower}, Theorem \ref{th:upper} and Corollary \ref{cor:BHI} to some classes of PDEs for which we easily can find solutions of the auxiliary ODIs in \eqref{eq:ODE} and \eqref{eq:ODE2} and thereby conclude upper and lower decay estimates and boundary Harnack inequalities.  
We begin with going through the simple case of zero lower order terms in detail, first with uniform ellipticity and then when ellipticity possibly blows up near the boundary. 

%\bigskip

\subsubsection{The case $\Phi^- \equiv \Phi^+ \equiv 0$ with uniform ellipticity}
\label{sec:example_0}

%\begin{example}
%...
%\end{example}
%
%\noindent
%{\bf Example 1:  $\Phi^- \equiv \Phi^+ \equiv 0$.}
In this case we consider solutions of \eqref{eq:main-nonlinear} 
with $F$ satisfying, for some constants $0 < \lambda \leq \Lambda$, 
\begin{align*}
\mathcal{P}^-_{\lambda,\Lambda}(X) 
\leq F(x,s,p,X) \leq
\mathcal{P}^+_{\lambda,\Lambda}(X), 
\end{align*}
whenever $x\in \mathbb{R}^n, s \in \mathbb{R}_+, p \in \mathbb{R}^n, X \in \mathbb{S}^n$. 
When $\lambda = 1 = \Lambda$ we have only the Laplace equation and then it is well known that solutions decay at the same rate as the distance function in $C^{1,1}$-domains.
We now intend to derive the similar estimates in the above more general setting, and to do so it is, according to our Corollary \ref{cor:BHI}, 
enough to find classical solutions of the ODIs \eqref{eq:ODE} and \eqref{eq:ODE2} and prove the corresponding estimates for those solutions.

Let $\Omega,r_0, w, r, u_1, u_2$ and $\Gamma_{w,r}$ be as in Corollary \ref{cor:BHI} and observe that, with $K = \frac{\Lambda n}{\lambda r}$,
\begin{align}\label{eq:first_case_solutions}
\check h(t) = m \frac{e^{Kt} - 1}{e^{Kr} - 1} \quad \text{and} \quad \hat h(t) = M \frac{1 - e^{-Kt}}{1 - e^{-Kr}} 
\end{align}
are increasing solutions of the ODIs \eqref{eq:ODE} and \eqref{eq:ODE2}, respectively, with $\check h(0) = 0 = \hat h(0)$, and $\check h(r) = m := \min_{i = 1,2}\inf_{\Gamma_{w,r}} u_i$ and $\hat h(r) = M := \max_{i=1,2}\sup_{B(w,6r)\cap \Omega} u_i$.
Moreover, we have the estimates
$$
\frac{ \Lambda n}{\lambda \left( e^{\frac{\Lambda n}{\lambda}} - 1\right)} \frac{m t}{r}  =  \frac{m K t}{e^{Kr} - 1} \leq \check  h(t) \leq
%\quad \text{and} \quad
\hat h(t) \leq \frac{M K t}{1 -e^{-Kr}}  = \frac{\Lambda n}{\lambda\left(1 - e^{-\frac{\Lambda n}{\lambda}}\right)} \frac{M t}{r}  
$$
for $t\in [0,r]$, so that Corollary \ref{cor:BHI} gives 
\begin{align}\label{eq:khabooomex1}
\frac{ \Lambda n}{\lambda \left( e^{\frac{\Lambda n}{\lambda}} - 1\right)} \frac{m d(x, \partial\Omega) }{r} 
\leq u_i(x) \leq  
\frac{\Lambda n}{\lambda\left(1 - e^{-\frac{\Lambda n}{\lambda}}\right)} \frac{M d(x, \partial\Omega) }{r}\quad 
%\text{for} \quad x \in \Omega \cap B(w, r), \quad i = 1,2.
\end{align}
for $x \in \Omega \cap B(w, r),  i = 1,2$.
Furthermore,
$$
1 \leq \frac{\hat h(t)}{\check h(t)} \leq \frac{M}{m} \frac{e^{\frac{\Lambda n}{\lambda}} - 1}{1 - e^{-\frac{\Lambda n}{\lambda}}} = \frac{M}{m} e^{\frac{\Lambda n}{\lambda}} \quad \text{whenever} \quad t \in (0,r)
$$
and hence we also derive the boundary Harnack inequality 
$$
e^{-\frac{\Lambda n}{\lambda}} \frac{m}{M}   
\leq \frac{u_1(x)}{u_2(x)} \leq 
\frac{M}{m}  e^{\frac{\Lambda n}{\lambda}} 
\quad \text{whenever} \quad x \in \Omega \cap B(w, r).
$$
We remark that the $p$-Laplace equation can be handled by the standard ellipticity case considered here; see e.g. the related works \cite{AJ17,L22,LOT20} for details.

%%%%%%%%%%%%%%%%%
%%%%%%%%%%%%%%%%%
%%%%%%%%%%%%%%%%%

\subsubsection{The case $\Phi^- \equiv \Phi^+ \equiv 0$ with vanishing ellipticity}
\label{sec:example_0_vanish}

We continue the above example but now we let the ellipticity vanish near the boundary by setting $\lambda(t) = t^a$ and $\Lambda$ constant (or, as $\Phi^+ \equiv \Phi^-\equiv 0$, it is equivalent to $\Lambda(t) = t^{-a}$ and $\lambda$ constant) for some $a \in (0,1)$. 
According to \eqref{eq:ass_drift_sub} we thus allow for equations in which ellipticity vanish everywhere on the boundary ($t = d(x,\partial \Omega)$), or just at one point ($t = d(x,w)$), when deriving a lower estimate.  
For an upper estimate we need to assume \eqref{eq:ass_drift_super}, and since $\lambda(t) = t^a$ is increasing, we may only allow for ellipticity to vanish at one point ($t = d(x,w)$).

With the help of Corollary \ref{cor:BHI}, 
it is sufficient to prove the intended results for classical solutions $\check h$ and $\hat h$ of
\begin{align*}
\check h'' =  \frac{n \Lambda}{r t^a} \check h', \qquad \hat h'' = - \frac{n \Lambda}{r t^a} \hat h', \qquad 0 < t < r.
\end{align*}
If $0 < a < 1$ then
$$
\log{\check h'(t)} - \log{\check h'(0)} = K\frac{t^{1-a}}{1-a}\quad \text{with}\quad K = \frac{n\Lambda}{r},
$$
$$
\check h'(t)  = \check h'(0) e^{K\frac{t^{1-a}}{1-a}},
$$
$$
\check h(t)  = \check h'(0) \int_0^t e^{K\frac{s^{1-a}}{1-a}} ds +  \check h(0),
$$
so that with $\check h(0) = 0 = \hat h(0)$, $i = 1,2$, we have 
$$
\check h(t) = \check h'(0) \int_0^t e^{\frac{K}{1-a} s^{1-a}} ds, \qquad \hat h(t) = \hat h'(0) \int_0^t e^{-\frac{K}{1-a} s^{1-a}} ds,
$$
and with $\check h(r) = m := \min_{i=1,2}\inf_{\Gamma_{w,r}}u_i$ and $\hat h(r) = M := \max_{i=1,2}\sup_{B(w,6r)\cap \Omega} u_i$, we obtain the increasing solutions 
$$
\check h(t) = m \;\frac{\int_0^t e^{\frac{K}{1-a} s^{1-a}} ds}{\int_0^r e^{\frac{K}{1-a} s^{1-a}} ds},\qquad  \hat h(t) = M \;\frac{\int_0^t e^{-\frac{K}{1-a} s^{1-a}} ds}{\int_0^r e^{-\frac{K}{1-a} s^{1-a}} ds}.
$$
Estimating the integrals gives
$$
e^{-\frac{\Lambda n}{(1-a)r^a}} \frac{m t}{r} \leq \frac{m t}{\int_0^r e^{\frac{K}{1-a} s^{1-a}} ds} \leq 
\check h(t)  
%\quad \text{and} \quad 
\leq
\hat h(t) \leq 
\frac{M t}{\int_0^r e^{-\frac{K}{1-a} s^{1-a}} ds} \leq  
e^{\frac{\Lambda n}{(1-a)r^a}} \frac{M t}{r} 
$$
for $t\in [0,r]$, so that Corollary \ref{cor:BHI} gives 
\begin{align*}
 e^{-\frac{\Lambda n}{(1-a)r^a}} \frac{m d(x, \partial\Omega) }{r} 
\leq u_i(x) \leq  
 e^{\frac{\Lambda n}{(1-a)r^a}} \frac{M d(x, \partial\Omega) }{r}\quad 
%\text{for} \quad x \in \Omega \cap B(w, r), \quad i = 1,2.
\end{align*}
for $x \in \Omega \cap B(w, r),  i = 1,2$.
Furthermore,
$$
1 \leq \frac{\hat h(t)}{\check h(t)} \leq \frac{M}{m} e^{\frac{2 \Lambda n}{(1-a) r^a}} \quad \text{whenever} \quad t \in (0,r)
$$
and hence we again derive the boundary Harnack inequality 
$$
 e^{-\frac{2 \Lambda n}{(1-a) r^a}} \frac{m}{M}
\leq \frac{u_1(x)}{u_2(x)} \leq 
\frac{M}{m} e^{\frac{2 \Lambda n}{(1-a) r^a}}
\quad \text{whenever} \quad x \in \Omega \cap B(w, r).
$$
The fact that constants blow up as $r \to 0$ is not surprising, since then the ellipticity vanishes.

%%%%%%%%%%%%%%%%%
%%%%%%%%%%%%%%%%%
%%%%%%%%%%%%%%%%%

\subsubsection{The case $\Phi^- \equiv -1, \Phi^+ \equiv 1$ with uniform ellipticity}
\label{sec:example_-1,1}

In this case we consider solutions of \eqref{eq:main-nonlinear} 
with $F$ satisfying
\begin{align}\label{eq:ex3equa}
-1 + \mathcal{P}^-_{\lambda,\Lambda}(X) 
\leq F(x,s,p,X) \leq
1 + \mathcal{P}^+_{\lambda,\Lambda}(X) 
\end{align}
whenever $x\in \mathbb{R}^n, s \in \mathbb{R}_+, p \in \mathbb{R}^n, X \in \mathbb{S}^n$, for some constants $0 < \lambda \leq \Lambda$. 
In what follows, let $\Omega,r_0, w, r, u_1, u_2$ and $\Gamma_{w,r}$ be as in Corollary \ref{cor:BHI} and put $K = \frac{\Lambda n}{\lambda r}$. 
According to our machinery, it suffices to prove the intended results for solutions of the ODIs 
\begin{align*}
\check h'' = \frac1\lambda + K \check h', \quad  0 < t < r, \quad \check h(0) = 0, \quad \check h(r) = m,
\end{align*}
and
\begin{align*}
\hat h'' = -\frac1\lambda - K \hat h', \quad  0 < t < r, \quad \hat h(0) = 0, \quad \hat h(r) = M.
\end{align*}
Standard calculations give 
\begin{align*}
\check h(t) = \left( m + \frac{r^2}{n\Lambda}\right) \frac{e^{K t} - 1}{e^{Kr}  - 1} - \frac{r}{n \Lambda} t, \quad 
\hat h(t) = \left( M + \frac{r^2}{n\Lambda}\right) \frac{1 - e^{-K t}}{1 - e^{-Kr}} - \frac{r}{n \Lambda} t,
\end{align*}
and concerning monotonicity, we have
\begin{align}\label{eq:jomenjahaomuärstortnoginnisåklart}
\check h'(t) > 0& \quad \text{for all}\quad t \in [0,r]\quad \notag
\Longleftrightarrow \quad \frac{m}{r^2} > \alpha, \\ 
\hat h'(t) > 0& \quad  \text{for all}\quad t \in [0,r]\quad \Longleftrightarrow\quad \frac{M}{r^2} > \alpha, \\ \notag
\text{where}& \quad
\alpha  = 
\frac{\lambda}{(\Lambda n)^2} \left(e^{\frac{n\Lambda}{\lambda}} - 1\right) - \frac{1}{\Lambda n}.
\end{align}
Since we have the freedom to choose any $M \geq \sup_{B(w,6r)\cap \Omega} u_i$, $i = 1,2$, 
the condition on $\hat h$ in \eqref{eq:jomenjahaomuärstortnoginnisåklart} will never violate the existence of an upper bound. 
However, we are required to choose $m \leq \inf_{\Gamma_{w,r}} u_i$, $i = 1,2$, and therefore we obtain a lower bound only if $\frac{\inf_{\Gamma_{w,r}} u_i}{r^2} > \alpha$. 
To summarize, Theorem \ref{th:lower} gives the following:
\begin{align*}%\label{}
\text{If} \quad \frac{\inf_{\Gamma_{w,r}}u_i}{r^2} > \alpha,\; i = 1,2, \quad \text{then} \quad \check h'(0) d(x, \partial\Omega) \leq  u_i(x) 
\quad \text{whenever} \quad x \in \Omega \cap B(w, r). 
\end{align*}
%
%for $i = 1,2$.
Moreover, Theorem \ref{th:upper} gives, for $i = 1,2$,
\begin{align*}%\label{}
u_i(x) \lesssim d(x, \partial\Omega) 
\quad \text{whenever} \quad x \in \Omega \cap B(w, r), %i = 1,2,
\end{align*}
and, furthermore, since
$$
\text{if} \quad \frac{\inf_{\Gamma_{w,r}}u_i}{r^2} > \alpha,\; i=1,2, \quad \text{then} \quad 1 \leq \frac{\hat h(t)}{\check h(t)} \lesssim 1 \quad \text{whenever} \quad t \in (0,r),
$$
Corollary \ref{cor:BHI} implies the following boundary Harnack inequality: 
$$
\text{If} \quad \frac{\inf_{\Gamma_{w,r}}u_i}{r^2} > \alpha,\;i=1,2, \quad \text{then} \quad\frac{u_1(x)}{u_2(x)} \approx 1 
\quad \text{whenever} \quad x \in \Omega \cap B(w, r).
$$

We remark that the extra condition on the lower estimate 
can easily be justified by considering the classical problem
\begin{align*}%\label{eq:Laplace1}
\Delta u = 1 \quad \text{in} \quad \{x \in \mathbb{R}^n : x_n > 0\} \quad \text{with} \quad u(x',0) = 0.
\end{align*}
Since for any $r > 0$ the function
\begin{align*}
    u(x) = \frac{x_n^2}2 + \left(\frac{u(r)}{r} - \frac{r}2\right) x_n 
\end{align*}
solves the problem, the decay estimates should switch behavior as the factor of $x_n$ switches sign. 
In particular, if $\frac{u(r)}{r^2} > \frac{1}{2}$ then $u(x) \approx x_n$ near the boundary $x_n = 0$, 
but otherwise such a lower estimate obviously fails. 
We also remark that -1 and 1 as the upper and lower bound in \eqref{eq:ex3equa} can be switched to constants without changing the nature of the above conclusions, 
and for the general class of equations considered in this example, the theorems on lower decay estimates in \cite[Section 5]{LOT20} will not apply due to assumption $(\phi_A)$ %and $(\phi_B)$ 
in that paper.

%%%%%%%%%%%%%%%%%
%%%%%%%%%%%%%%%%%
%%%%%%%%%%%%%%%%%

\subsubsection{The case $\Phi^-(t,s,|p|) = - |p|^k$, $k\geq 1$, with uniform ellipticity}
\label{sec:example_pk}

In this case we consider solutions of \eqref{eq:main-nonlinear} 
with $F$ satisfying
\begin{align*}
-|p|^k + \mathcal{P}^-_{\lambda,\Lambda}(X) 
\leq F(x,s,p,X) %\leq
\end{align*}
whenever $x\in \mathbb{R}^n, s \in \mathbb{R}_+, p \in \mathbb{R}^n, X \in \mathbb{S}^n$, for some constants $0 < \lambda \leq \Lambda$. 
Since $F$ is unbounded from above a lower estimate is not possible. 
If $k\leq 2$ then we know from \cite{LOT20} that an upper estimate $u \lesssim d(x,\partial\Omega)$ holds in general, but otherwise, a counterexample in \cite[Section 5]{LOT20} disproves an upper estimate in the general case. 
Thus, it remains to investigate if there are situations in which it is possible to derive an upper estimate also for $k > 2$.
Let $\Omega,r_e, w, r$ and $u$ be as in Theorem  \ref{th:upper}, introduce $\hat f = \hat h'$ and $K = \frac{n\Lambda}{r\lambda}$. 
To prove an upper estimate for $k > 2$, we take off from  \eqref{eq:ODE2} and solve the initial value problem
$$
\frac{d\hat f}{dt} = - \frac{\hat f^k}{\lambda} - K \hat f, \quad t \in (0,r)  \quad \text{with} \quad \hat f(0) = \nu>0,
$$
yielding the solution (see \cite[Section 3.2]{L22} for a derivation), 
$$
\hat f_\nu(t) =  \left(
\frac{\lambda K}{e^{(k-1) K t} \left(\frac{\lambda{K}}{\nu^{k-1}} + 1\right) - 1}\right)^\frac{1}{k-1}.
$$
Now, $\hat f_\nu$ is decreasing so that for any $\nu > 0$,
$$
\hat h_\nu(t) = \int_0^t \hat f_\nu(s) ds
$$
is concave and, furthermore, $\hat h_\nu(r)$ is increasing in $\nu$. 
Thus, as long as there is $\nu > 0$ so that $\hat h_\nu(r) \geq \sup_{B(w,6r)\cap \Omega} u$ we obtain an upper estimate. This means that we obtain an estimate only when the subsolution $u$ is not too large a small distance from the boundary. In particular, we have derived the following:
\begin{align}\label{eq:upperboundexsk1}
\text{If} \quad \hat h_\nu(r)\, \geq \sup_{B(w,6r)\cap \Omega} u \quad \text{then} \quad 
u(x) \leq \hat h_\nu(d(x,\partial \Omega)) \leq \nu d(x,\partial \Omega) %\quad \text{whenever} \quad x \in \Omega \cap B(w,r).
\end{align}
whenever $x \in \Omega \cap B(w,r)$.

Next, we observe that $\hat f_\nu$ and $\hat h_\nu$ are maximized when $\nu \to \infty$, and
$$
\hat f_\nu(t) \to \hat f_\infty(t) =  
\left(\frac{\lambda K}{e^{(k-1) K t}  - 1}\right)^\frac{1}{k-1}, %\quad \hat h_\nu(t) \to \hat h_\infty(t),
$$
and  for $0 < \epsilon \leq t \leq r$ we have
$$
\left( \frac{\lambda K r}{e^{(k-1) K r} - 1}  \right)^\frac{1}{k-1}
%\check c 
\int_\epsilon^t {s^\frac{1}{1-k}}ds \leq \hat h_\infty(t)-\hat h_\infty(\epsilon) \leq
%\hat c 
\left(\frac{\lambda}{k-1} \right)^\frac{1}{k-1}
\int_\epsilon^t {s^\frac{1}{1-k}} ds. 
$$
Thus, for $0 \leq t \leq r$,
$$
\check c \,t^\frac{k-2}{k-1} \leq \hat h_\infty(t) \leq \hat c \,t^\frac{k-2}{k-1},  
$$
with 
\begin{align}\label{eq:constantsinexamplesk}
\check c = \frac{k-1}{k-2}\left( \frac{n\Lambda}{e^{(k-1) \frac{n\Lambda}{\lambda} } - 1}  \right)^\frac{1}{k-1}
, \quad 
\hat c = \frac{k-1}{k-2}\left(\frac{\lambda}{k-1} \right)^\frac{1}{k-1}.
\end{align}
Therefore, we obtain the following upper estimate:
\begin{align}\label{eq:upperboundexsk2}
\text{If} \quad \check c \,r^\frac{k-2}{k-1}\, \geq \sup_{B(w,6r)\cap \Omega} u \quad \text{then} \quad 
u \leq \hat c\, d(x,\partial \Omega)^\frac{k-2}{k-1} 
\end{align}
whenever $x \in \Omega \cap B(w,r)$ and where $\check c$ and $\hat c$ are as in \eqref{eq:constantsinexamplesk}.
In conclusion, estimate \eqref{eq:upperboundexsk1}  reveals that, when $\sup_{B(w,6r)\cap \Omega} u$ is small, then  subsolutions decay no slower than proportional to the distance function, but the constant of proportion ($\nu$) explodes as $ \sup_{B(w,6r)\cap \Omega} u$ increases. 
Estimate \eqref{eq:upperboundexsk2} reveals a power-law upper bound with constant independent of $\sup_{B(w,6r)\cap \Omega} u$. 

To remark on the sharpness of these results, 
let $r > 0$ and $k > 2$ be given 
and consider the following boundary value problem
\begin{align}\label{eq:oneDk}
u''(x) = - |u'(x)|^k, \quad x \in (0,r), \quad u(0) = 0, \quad u(r) = M > 0.
\end{align}
Standard calculations show that the function 
$$
w_\gamma(x) =  \frac{(k-1)^\frac{k-2}{k-1}}{k-2} \left( \left(x + \gamma \right)^\frac{k-2}{k-1} - \gamma^\frac{k-2}{k-1}\right) 
$$
is a classical solution, 
where $\gamma \geq 0$ can be determined by $u(r) = M$ if $M$ is small enough. 
We observe that 
when $x$ approaches the boundary at $x = 0$, then the concave solution $w_\gamma$ decays proportionally to $x$, but with constant exploding as $\gamma \to 0$. This is in line with estimate \eqref{eq:upperboundexsk1}. 
Moreover, when $w_\gamma$ is maximized for $\gamma = 0$ we obtain the power-law decaying solution with 
exponent $\frac{k-2}{k-1}$, which agrees with the bound in \eqref{eq:upperboundexsk2}.

%%%%%%%%%%%%%%%%%%%%%%%%%%%%%%%%%%%%%%%%%%%%
%%%%%%%%%%%%%%%%%%%%%%%%%%%%%%%%%%%%%%%%%%%%
%%%%%%%%%%%%%%%%%%%%%%%%%%%%%%%%%%%%%%%%%%%%
%%%%%%%%%%%%%%%%%%%%%%%%%%%%%%%%%%%%%%%%%%%%
%%%%%%%%%%%%%%%%%%%%%%%%%%%%%%%%%%%%%%%%%%%%
%%%%%%%%%%%%%%%%%%%%%%%%%%%%%%%%%%%%%%%%%%%%

\section{H\"older continuity of $u/v$ near the boundary}
\label{sec:holder}

\setcounter{theorem}{0}
\setcounter{equation}{0}

In this section, we study the conditions under which the quotient $u/v$, for which we established bounds in the previous section, is also H\"older continuous near the boundary. 
We first demonstrate, in Section \ref{sec:basic}, that global $C^{1,\alpha}$-estimates combined with suitable boundary growth estimates yield the desired continuity.
Secondly, we apply the results to a class of fully nonlinear PDEs for which global $C^{1,\alpha}$ is available (Section \ref{sec:fully_nonlinear}). In Section \ref{sec:p(x)} and Section \ref{sec:inf-lap} we consider reflection arguments and local $C^{1,\alpha}$-estimates to obtain the similar estimates for $p(x)$-harmonic functions and for planar infinity-harmonic functions near locally flat boundaries. 
Finally, in Section \ref{sec:applications-of-u/v-reg} we derive some corollaries of $u_1/u_2 \in C^{\alpha}$-estimates in unbounded domains.  

\subsection{Basic lemmas}
\label{sec:basic}

We begin with some basic lemmas for a differentiable function $f$ vanishing on the boundary of a domain satisfying the interior ball condition. 
By differentiability of the distance function $d(x,\partial \Omega)$ near the boundary, the quotient $f/d(x,\Omega)$ becomes differentiable. The following Lemma shows that if the gradient of $f$ is H\"older continuous, then so is the quotient $f/d(x,\Omega)$.  
\begin{lemma}\label{le:first}
Let $\Omega$ satisfy the ball condition with radius $r_0$, $r \in (0,r_0]$, $z \in \partial \Omega$, $f \in C^{1,\alpha}$ so that 
\begin{align}\label{eq:C1_alpha}
|\nabla f(x) - \nabla f(y) | \leq  c_1 |x-y|^{\alpha}, \quad \text{and}\quad |\nabla f(x)| \leq c_1,  
\end{align}
whenever $x,y \in B(z,r)\cap \Omega$, and $f = 0$ on $B(z,r)\cap\partial \Omega$. 
Then there exists $c$, depending only on $c_1$ and $r_0$, such that 
$$
\left|\frac{f(x)}{d(x,\partial\Omega)} - \frac{f(y)}{d(y,\partial\Omega)}\right| \leq  c |x - y|^{\alpha},
$$
whenever $x,y \in B(z,r/c)\cap \Omega$.
\end{lemma}
\noindent
\begin{proof}
Let $(x,y)\in B(z, r/c) \cap \Omega$ for some $c \geq 1$ to be increased later, %, where $x= (x',x_n)$ and $y=(y', y_n)$.
and assume w.l.o.g. that $d(x,\partial\Omega)>d(y,\partial\Omega)$. 
Let $y_0 \in \partial\Omega$ be such that $d(y,y_0) = d(y,\partial\Omega)$ and define $w \in \Omega$ so that $d(x,\partial\Omega) = d(w,\partial \Omega)$. Thus, $w = y_0 + d(x,\partial\Omega) \eta_y$ with $\eta_y = \frac{y-y_0}{d(y,\partial \Omega)}$. 
Uniqueness of $y_0$ follows from the $C^{1,1}$ regularity of the boundary, for $c$ large enough. 
We will show the required estimate for
$$
\left|\frac{f(x)}{d(x,\partial\Omega)} - \frac{f(w)}{d(w,\partial\Omega)}\right| \qquad \mbox{and} \qquad \left|\frac{f(w)}{d(w,\partial\Omega)} - \frac{f(y)}{d(y,\partial\Omega)}\right|
$$ 
separately, then the result follows from the triangle inequality.
Starting with the former term, \eqref{eq:C1_alpha} and the interior ball condition of the domain yield, with $x_0 \in \partial\Omega$ so that $d(x,x_0) = d(x,\partial\Omega)$, 
\begin{multline}\label{est1}
\left|\frac{f(x)}{d(x,\partial\Omega)}-\frac{f(w)}{d(w,\partial\Omega)}\right|\\[10pt]
=
\frac{1}{d(x,\partial\Omega)}\left|\int_0^{d(x,\partial\Omega)}
\left(\langle \nabla f(x_0 + t\eta_x) - \nabla f(y_0 + t\eta_y),\eta_x\rangle + \langle\nabla f(y_0 + t\eta_y),\eta_x -\eta_y\rangle\right)\right|\,dt\\[10pt]
\leq
\frac{1}{d(x,\partial\Omega)}\int_0^{d(x,\partial\Omega)}\left(\left|\nabla f(x_0 + t\eta_x)-\nabla f(y_0 + t\eta_y)\right| +  c |x-y| \right)\,dt\\[10pt]
\leq \frac{1}{d(x,\partial\Omega)}\int_0^{d(x,\partial\Omega)}c\left|x_0 + t\eta_x - y_0 - t\eta_y\right|^{\alpha}\,dt + c |x-y|\\[10pt]
%
%\leq\frac{1}{d(x,\partial\Omega)}\int_0^{d(x,\partial\Omega)}2c_K\left|x-w\right|^{\alpha_K}\,dt
\leq c \left|x-y\right|^{\alpha},
\end{multline}
%
%\bigskip 
%
\noindent
for $c$ large enough, depending only on $r_0$ and $c_1$. 
In the last line we have used the ball condition for $\Omega$ to obtain 
$|x_0+t\eta_x-y_0-t\eta_y|\leq c|x-y|.$

For the second term the mean value theorem of vector calculus and the $C^{1,1}$ regularity of the boundary 
give 
\begin{equation*}
\left|\frac{f(w)}{d(w,\partial\Omega)}-\frac{f(y)}{d(y,\partial\Omega)}\right|
\leq\left|\max_{t\in L}\frac{\grad f(t)d(t,\partial\Omega)-f(t)\grad d(t,\partial\Omega)}{d(t,\partial\Omega)^2}\cdot(w-y)\right|,
\end{equation*}
\noindent
where $L$ is the straight line between $y$ and $w$. 
Since $f(y_0)=0$ 
the mean value theorem implies the existence of
$\xi\in[y_0,t]$ such that $f(t)=\grad f(\xi)\cdot(t-y_0)$, hence 
\begin{align*}
 \left|\frac{f(w)}{d(w,\partial\Omega)}-\frac{f(y)}{d(y,\partial\Omega)}\right|%\\[10pt]
  &\leq\left|\max_{t\in L}\frac{\grad f(t)d(t,y_0)-\left(\grad f(\xi)\cdot(t-y_0)\right)\grad d(t,y_0)}{d(t,y_0)^2}\cdot(w-y)\right|\\
%
% =\left|\max_{t\in L}\frac{\grad f(t)d(t,y_0) \cdot(w-y) -\left(\grad f(\xi)\cdot(t-y_0)\right) \grad d(t,y_0)\cdot(w-y)}{d(t,y_0)^2}\right|\\
%
% =\left|\max_{t\in L}\frac{\grad f(t)d(t,y_0) \cdot(w-y) -\left(\grad f(\xi)\cdot(t-y_0)\right) |w-y|}{d(t,y_0)^2}\right| \\
%
 &=\left|\max_{t\in L}\frac{ |w-y| \grad f(t) \cdot(t-y_0) -|w-y|\grad f(\xi)\cdot(t-y_0) }{d(t,y_0)^2}\right|\\
 &\leq\max_{t\in L}\frac{ |w-y|  \cdot |\grad f(t)  - \grad f(\xi)| }{d(t,y_0)}\\
&\leq\max_{t\in L}c\frac{|t-\xi|^{\alpha}}{|t-y_0|}|w-y|.
\end{align*}
\noindent
If $\frac{|w-y_0|}{2}\leq|y-y_0|$ then $|w-y|\leq|y-y_0|\leq|t-y_0|$, and we always have $|t-\xi|\leq|t-y_0|$, hence
\begin{equation*}
    \left|\frac{f(w)}{d(x,\partial\Omega)}-\frac{f(y)}{d(y,\partial\Omega)}\right|
    \leq\max_{t\in L} c_1 \frac{|w-y|}{|t-y_0|^{1-\alpha}}\leq c_1|w-y|^{\alpha}\leq 2c_1|x-y|^{\alpha},
\end{equation*}
\noindent
where in the last step we have used $|w-y|\leq2|x-y|$ for $c$ large enough.

If $|y-y_0|<\frac{|w-y_0|}{2}$ we can immediately
use the mean value theorem ($f(y_0)=0$ for all $y_0\in\partial\Omega$), thus there exists $\xi\in L_1$ and $\eta\in L_2$ ($L_1$ is the straight line between $y_0$ and $w$ and $L_2$ is the straight line between $y_0$ and $y$) such that
\begin{multline*}
  \left|\frac{f(w)}{d(w,y_0)}-\frac{f(y)}{d(y,y_0)}\right|=\left|\frac{\langle w-y_0, \grad f(\xi)\rangle}{d(w,y_0)}-\frac{\langle y-y_0, \grad f(\eta)\rangle}{d(y,y_0)}  \right|\\[10pt]
  =|\langle\grad f(\xi)-\grad f(\eta),\eta_y\rangle|\leq |\grad f(\xi)-\grad f(\eta)||\eta_y|\leq c_1|\xi-\eta|^{\alpha}.
\end{multline*}
\bigskip
Since $|w-y|\geq\frac{|w-y_0|}{2}$ we have $|\xi-\eta|^{\alpha}\leq2^\alpha\left(\frac{|w-y_0|}{2}\right)^\alpha\leq 2^\alpha|w-y|^\alpha$, hence
\begin{equation}\label{est2}
    \left|\frac{f(w)}{d(w,y_0)}-\frac{f(y)}{d(y,y_0)}\right|\leq 2^{ \alpha}  c_1|w-y|^{\alpha}
    \leq 2 c_1|w-y|^{\alpha}
    \leq c |x-y|^{\alpha},
\end{equation}
\noindent
where we have used %that $0<\alpha_K\leq 1$ and the fact 
the interior ball condition for $\Omega$ and therefore $|w-y|\leq 2|x-y|$ for $c$ large enough.
Now the estimates~\eqref{est1} and~\eqref{est2} imply 
\begin{multline*}
    \left|\frac{f(x)}{d(x,\partial\Omega)}-\frac{f(y)}{d(y,\partial\Omega)}\right|\leq\left|\frac{f(x)}{d(x,\partial\Omega)}-\frac{f(w)}{d(w,\partial\Omega)}\right|+\left|\frac{f(w)}{d(w,\partial\Omega)}-\frac{f(y)}{d(y,\partial\Omega)}\right|\\[10pt]
    \leq %4c_1|x-y|^{\alpha}+ 4c_1|x-y|^{\alpha}=8c_1
    c|x-y|^{\alpha},
\end{multline*}
which completes the proof.
\end{proof}

%%%%%%%%%%%%%%%%%%%%%%%%%%%%%%%%
%%%%%%%%%%%%%%%%%%%%%%%%%%%%%%%%

\bigskip

The growth estimate \eqref{eq:dist-vanish} below, together with Lemma \ref{le:first}, readily implies the following simple lemma:
\begin{lemma}\label{le:second}
Let $\Omega$, $r_0$, $r$, $z$ and $f$ be as in Lemma \ref{le:first} and
assume that $g$ satisfies the same assumptions as $f$ and in addition 
\begin{align}\label{eq:dist-vanish}
0 < C_L \leq \frac{g(x)}{d(x,\partial \Omega)} \quad \text{and} \quad 0\leq \frac{f(x)}{d(x,\partial \Omega)} \leq C_U
\end{align}
whenever $x,y \in \Omega\cap B(z,r)$.
%\eqref{eq:dist-vanish} in $B(z,R)\cap\Omega$. 
Then it holds that
\begin{equation*}
\left|\frac{d(x,\partial\Omega)}{g(x)} - \frac{d(y,\partial\Omega)}{g(y)}\right| \leq  \hat c |x - y|^{\alpha} \quad \text{and} \quad 
\left|\frac{f(x)}{g(x)} - \frac{f(y)}{g(y)}\right| \leq  \check c |x - y|^{\alpha}
\end{equation*}
%\bigskip
%
whenever $x,y \in B(z,r/c)\cap \Omega$, and where $\hat c = c C_L^{-2}$ and $\check c=c (C_L+C_U)/C_L^2$ with $c$ from Lemma \ref{le:first}. 
\end{lemma}
\noindent
\begin{proof}
Let $x,y\in B(z,r/c)\cap\Omega$. Then from  Lemma~\ref{le:first} and inequality~\eqref{eq:dist-vanish} we can deduce
\begin{multline}\label{part1}
    \left|\frac{d(x,\partial\Omega)}{g(x)}-\frac{d(y,\partial\Omega)}{g(y)}\right|=\frac{d(x,\partial\Omega)\,d(y,\partial\Omega)}{g(x)g(y)}\left|\frac{g(y)}{d(y,\partial\Omega)}-\frac{g(x)}{d(x,\partial\Omega)}\right|\\[10pt]
    \leq c\frac{d(x,\partial\Omega)\,d(y,\partial\Omega)}{g(x)g(y)}\left|x-y\right|^{\alpha}
    \leq c C_L^{-2}|x-y|^{\alpha}.
\end{multline}
%\bigskip
%
\noindent
For the second part of the lemma we make use of~Lemma~\ref{le:first}, estimate~\eqref{eq:dist-vanish} and~\eqref{part1} to see that
\begin{multline*}
    \left|\frac{f(x)}{g(x)}-\frac{f(y)}{g(y)}\right|%=\left|\frac{f(x)}{g(x)}+\frac{d(x,\partial\Omega) f(y)}{g(x)d(y,\partial\Omega)}-\frac{d(x,\partial\Omega)f(y)}{g(x)d(y,\partial\Omega)}-\frac{f(y)}{g(y)}\right| \\[10pt]
    \leq\frac{d(x,\partial\Omega)}{g(x)}
    \left|\frac{f(x)}{d(x,\partial\Omega)}-\frac{f(y)}{d(y,\partial\Omega)}\right|+
    \frac{f(y)}{d(y,\partial\Omega)}\left|\frac{d(x,\partial\Omega)}{g(x)}-\frac{d(y,\partial\Omega)}{g(y)}\right|\\[10pt]
    \leq c C_L^{-1}|x-y|^{\alpha}+c C_U C_L^{-2}|x-y|^{\alpha}=\check c|x-y|^{\alpha},
\end{multline*}
%\bigskip
%
\noindent
where $\check c=c (C_L+C_U)/C_L^2$. 
\end{proof}

%%%%%%%%%%%%%%%%%%%%%%%%%%%%%%%%%%%%%%%%%%%%
%%%%%%%%%%%%%%%%%%%%%%%%%%%%%%%%%%%%%%%%%%%%
%%%%%%%%%%%%%%%%%%%%%%%%%%%%%%%%%%%%%%%%%%%%
%%%%%%%%%%%%%%%%%%%%%%%%%%%%%%%%%%%%%%%%%%%%
%%%%%%%%%%%%%%%%%%%%%%%%%%%%%%%%%%%%%%%%%%%%
%%%%%%%%%%%%%%%%%%%%%%%%%%%%%%%%%%%%%%%%%%%%
%%%%%%%%%%%%%%%%%%%%%%%%%%%%%%%%%%%%%%%%%%%%
%%%%%%%%%%%%%%%%%%%%%%%%%%%%%%%%%%%%%%%%%%%%
%%%%%%%%%%%%%%%%%%%%%%%%%%%%%%%%%%%%%%%%%%%%

\subsection{Fully nonlinear equations}
\label{sec:fully_nonlinear}

We now return to fully nonlinear elliptic equations in nondivergence form, as in \eqref{eq:main-nonlinear}. 
To ensure global $C^{1,\alpha}$ we assume here, in addition to the assumptions in Corollary \ref{cor:BHI}, that 
$F(x, 0, 0, 0) \equiv 0$, and that the following structural assumptions hold: 
\begin{align}\label{eq:ass-nonlinear1}
\mathcal{P}^-_{\lambda, \Lambda}(X - Y ) - \mu |p - q|(1 + |p| + |q|) -  \omega(|r - s|)\notag\\
\leq F(x, r, p, X) - F(x, s, q, Y ) \\
\leq 
%\Lambda \text{Trace} (Y)
\mathcal{P}^+_{\lambda, \Lambda}(X - Y ) + \mu |p - q|(1 + |p| + |q|) +  \omega(|r - s|),\notag \\
x\in \Omega, \; r,s \in \mathbb{R}, \; p,q \in \mathbb{R}^n, \; X, Y \in \mathbb{S}^n,\notag
\end{align}
where $0 < \lambda < \Lambda$, $\mu\geq0$, and $\omega:[0, \infty] \to [0, \infty]$, $\omega(0) =0$ is a modulus of continuity satisfying
\begin{align}\label{eq:ass_omega_near}
\int_{0}^{1} \frac{dt}{\omega(t)} = \infty.
\end{align}
In addition, to ensure the validity of the $C^{1,\alpha}$ estimate in the literature, 
we also assume through this subsection that $F(x,0,0,X)$ is uniformly continuous in $x$,
uniformly with respect to $X\in\mathbb S^n$.

The following boundary $C^{1,\alpha}$-regularity result is from \cite{N19} and is originally stated in a more general form:
\begin{lemma}\label{le:nonlinear_C1a}
Let $\Omega \subset \mathbb{R}^n$ be a $C^{1,1}$-domain, $F$ satisfy \eqref{eq:ass-nonlinear1}, %\eqref{eq:oscillation_F}, 
$w \in \partial \Omega$ and let $u$ be an $L^p$-viscosity solution, $p>n$, of \eqref{eq:main-nonlinear} in $B(w,r) \cap \Omega$, continuous in $\overline{\Omega}\cap B(w,r)$ satisfying $u = 0$ on $\partial \Omega \cap B(w,r)$. 
Then there exist $c$ and $\alpha \in (0,1]$ such that
$$
|Du(x) - Du(y)| \leq c |x - y|^{\alpha}\quad \text{and}\quad |Du(x)| \leq c,
$$
whenever $x,y \in B(w,r/c) \cap \Omega$.
\end{lemma}

\noindent
\begin{proof}
Observing that the structural condition \eqref{eq:ass-nonlinear1} imposed on $F$ implies the more general structural assumption in \cite{N19}, 
and that $F(x,0,0,X)$ is uniformly continuous in $x$,
uniformly with respect to $X\in\mathbb S^n$, this follows from \cite[Theorem 1.1 and Remark 1.2]{N19}. 
\end{proof}

\bigskip

We now apply our estimates from the previous section--specifically Corollary \ref{cor:BHI}--to show that assumption \eqref{eq:ass-nonlinear1} also implies the desired boundary growth rate. 
To ensure that $L^p$-viscosity (sub and super) solutions are also viscosity (sub and super) solutions, we may assume that $F$ is continuous in all its variables. 
We now prove the following lemma:

\begin{lemma}\label{le:nonlinear_bhi}
Let $\Omega \subset \mathbb{R}^n$ 
be a $C^{1,1}$-domain, $F$ satisfy \eqref{eq:ass-nonlinear1}, and $w \in \partial \Omega$. %$0 < r < r_0$. 
Assume that $u_1, u_2$ are positive $L^p$-viscosity solutions, $p>n$, of \eqref{eq:main-nonlinear} in $\Omega \cap B(w, 7r)$, continuous in $\overline{\Omega} \cap B(w, 7r)$
satisfying $u_1 = 0 = u_2$ on $\partial \Omega \cap B(w,6r)$.
Then there exists $c$ such that
\begin{equation*}
    c^{-1}{d(x,\partial\Omega)}\leq u_i(x)\leq c d(x,\partial\Omega),\quad
    \text{and}\quad c^{-1}\leq \frac{u_1(x)}{u_2(x)}\leq c, %\quad\forall\, x\in B(w, r/c), \quad i = 1,2.
\end{equation*}
whenever $x \in \Omega \cap B(w, r)$ and $i = 1,2.$
\end{lemma}

\noindent
\begin{proof}
For a lower bound, we first note that \eqref{eq:ass-nonlinear1} implies
\begin{align*}%\label{eq:tjohej}
F(x,s,p,X) \leq P^+_{\lambda, \Lambda}(X) + \mu (|p| + |p|^2) + \omega(s)
\end{align*}
whenever $s \in [0,\infty), x, p \in \mathbb{R}^n, X \in \mathbb{S}^n$, 
in which $\omega$ satisfies the Osgood condition \eqref{eq:ass_omega_near}.
The corresponding ODI leads us to the DE
$$
\check h''
=
\left(\frac{\mu}{\lambda} + \frac{n \Lambda}{r \lambda}\right)\check h'
+\frac{\mu}{\lambda}(\check h')^2
+\frac1\lambda\omega(\check h),
\quad
\check h(0)=0,\quad \check h'(0)=\nu>0.
$$
We claim that, for fixed $r > 0$, sufficiently small $\nu > 0$ ensures that $\check h$ exists on $[0, r]$
and that $\check h(r)$ is arbitrarily small.
Indeed, since $\omega$ is nonnegative and nondecreasing, we have
\begin{align}\label{eq:hallojjuss}
(\check h + \check h')'
\leq C \bigl(\check h + \check h' + \omega(\check h + \check h')\bigr)
\end{align}
as long as $\check h + \check h' \leq 1$,
where $C$ is independent of $\nu$.
We assume without loss of generality that $\omega(s)\geq s$ and 
observe that the Osgood condition \eqref{eq:ass_omega_near} implies
$$
\int_0^1\frac{ds}{s+\omega(s)}=\infty.
$$
Now \eqref{eq:hallojjuss} can be written, with $f(t) = \check h + \check h '$,
$$
f'(t) \leq C \left(f + \omega(f) \right), \; f(0) = \nu, \quad \Longrightarrow \quad t \geq \frac{1}{C}\int_\nu^{f(t)} \frac{ds}{s + \omega(s)}.
$$
Thus, the $t$ required for
$f(t)$ to reach any small value satisfies $t \uparrow \infty$ as $\nu\downarrow0$.
Sufficiently small $\nu$ therefore ensures existence of a solution $f(t)$ on $[0,r]$ so that 
$
\check h(r) \leq  m
$
for any prescribed $m>0$.
Moreover, $\check h''\geq 0 $ gives
$$
\check h(t)\geq\nu t,\quad 0\leq t\leq r.
$$
Choosing $0<m\leq\min_{i=1,2}\inf_{\Gamma_{w,r}}u_i$
provides the required lower barrier.

To establish an upper bound we observe that \eqref{eq:ass-nonlinear1}, 
$\omega(0) = 0$, and \eqref{eq:deg-ellipt-proper} imply  
$$
P^-_{\lambda, \Lambda}(X) - \mu (|p| + |p|^2) \leq F(x,0,p,X),
$$
whenever $x, p \in \mathbb{R}^n, X \in \mathbb{S}^n$. 
According to Corollary \ref{cor:BHI} it suffices to find an upper bound on an increasing classical solution of 
\begin{align}\label{eq:diff-nonlinear-bhi}
\hat h'' \leq - \left(\frac{\mu}{\lambda} + K\right) \hat h' - \frac{\mu}{\lambda} \hat h'^2 \quad 0< t< r, \quad \hat h(0) = 0, \quad \hat h(r) \geq \sup_{B(w,6 r)\cap \Omega} u_i, 
\end{align}
$i = 1,2$, where $K = \frac{n\Lambda}{r\lambda}$.
Let $\hat f = \hat h'$ and observe that then 
$$
\hat f(t) = \frac{(\mu+\lambda K)\nu}{(\mu+\lambda K)e^{(\mu/\lambda+K)t} 
+ \mu\nu\left(e^{(\mu/\lambda+K)t}-1\right)}
$$
solves the ODI in \eqref{eq:diff-nonlinear-bhi} with equality on $0 < t < r$ and satisfies $\hat f(0) = \nu$.
Moreover, $\hat f > 0$ is decreasing so that 
$$
\hat h(t) = \int_0^t \hat f(s) ds
$$
is increasing, concave, and it is not hard to see that $\hat h(r) \to \infty$ as $\nu \to \infty$.
Therefore, taking $\nu$ large enough, 
 $\hat h$ satisfies \eqref{eq:diff-nonlinear-bhi}, 
and in addition $\hat h(t) \leq \nu t$ on $[0, r]$.

Armed with the above solutions $\check h$ and $\hat h$, and the bounds of them on $[0,r]$, the lemma follows by an application of Corollary \ref{cor:BHI}. 
\end{proof}\\

We are now ready to state and prove that the ratio of two positive viscosity solutions, vanishing continuously on a portion of the boundary, is Hölder continuous.  

\begin{theorem}\label{th:u/v Hölder fully nonlinear}
Let $\Omega \subset \mathbb{R}^n$ be a $C^{1,1}$-domain,
$F$ satisfy \eqref{eq:ass-nonlinear1}, and $w \in \partial \Omega$. 
Assume that $u$ and $v$ are positive $L^p$-viscosity solutions, $p>n$, of \eqref{eq:main-nonlinear} in $\Omega \cap B(w, r)$, continuous in $\overline{\Omega} \cap B(w, r)$
satisfying $u = 0 = v$ on $\partial \Omega \cap B(w,r)$.
Then there exists $c$ such that 
\begin{equation*}%\label{eq:}
\left| \frac{u(x)}{v(x)} - \frac{u(y)}{v(y)}\right | \leq c  |x-y|^{\alpha}
\end{equation*}
whenever $x,y \in B(w, r/c) \cap \Omega$ and where $\alpha\in(0,1]$ is from Lemma \ref{le:nonlinear_C1a}. 
Moreover, the result holds with $u$ or $v$ replaced by $d(x, \partial \Omega)$ as well.  
\end{theorem}

\noindent
\begin{proof}
The result is an immediate consequence of Lemmas \ref{le:first}, \ref{le:second}, \ref{le:nonlinear_C1a} and \ref{le:nonlinear_bhi}.
\end{proof}\\

\noindent
We remark that the constants in Lemma \ref{le:nonlinear_C1a},  Lemma \ref{le:nonlinear_bhi} and Theorem \ref{th:u/v Hölder fully nonlinear} may comprise complicated dependence on the ingoing parameters. %solutions and the class of equations. 
To avoid technicalities, we have chosen to not go into such details in this general setting.

Even though the class of equations considered here is rather general, it does not cover the $\infty$-Laplace equation (which fails to be elliptic in the sense of \eqref{eq:ass-nonlinear1}) nor the variable exponent $p$-Laplace equation (which fails due to a non-Lipschitz dependence on the gradient). 
For both equations, the boundary growth estimates, $u \approx d$, 
are well known, but the $C^{1,\alpha}$ estimates are, to the best of our knowledge, only established in the interior of the domain, and only in planar domains for the $\infty$-Laplace equation. 
However, we can circumvent the lack of boundary regularity by applying a Schwarz reflection argument if the boundary is locally flat.  In the following subsections we will do so and thereby prove versions of Theorem \ref{th:u/v Hölder fully nonlinear} for $\infty$-harmonic functions and for $p(x)$-harmonic functions.

%%%%%%%%%%%%%%%%%%%%%%%%%%%%%%%%%%%%%%%%%%%%
%%%%%%%%%%%%%%%%%%%%%%%%%%%%%%%%%%%%%%%%%%%%
%%%%%%%%%%%%%%%%%%%%%%%%%%%%%%%%%%%%%%%%%%%%
%%%%%%%%%%%%%%%%%%%%%%%%%%%%%%%%%%%%%%%%%%%%
%%%%%%%%%%%%%%%%%%%%%%%%%%%%%%%%%%%%%%%%%%%%
%%%%%%%%%%%%%%%%%%%%%%%%%%%%%%%%%%%%%%%%%%%%
%%%%%%%%%%%%%%%%%%%%%%%%%%%%%%%%%%%%%%%%%%%%
%%%%%%%%%%%%%%%%%%%%%%%%%%%%%%%%%%%%%%%%%%%%
%%%%%%%%%%%%%%%%%%%%%%%%%%%%%%%%%%%%%%%%%%%%

\subsection{The $p(x)$-Laplace equation}
\label{sec:p(x)}

In this section we derive the analogue of Theorem \ref{th:u/v Hölder fully nonlinear} for solutions of the $p$-Laplace equation with a variable exponent, that is,
\begin{align}\label{eq:p(x)-lap}
\nabla  \cdot \left(|\nabla u|^{p(x)-2} \nabla u\right) = 0
\end{align}
in which $p:\Omega\to (1,\infty)$, $1<p^-\leq p(x)\leq p^+<\infty$ is a continuous function called a \emph{variable exponent}. 
We consider solutions in the following standard weak sense. 

\begin{definition}
A function $u\in W^{1,\,p(x)}(\Omega)$ is a weak subsolution (supersolution) in $\Omega$ if
\begin{align}\label{weakp}
    \int_{\Omega} |\grad u|^{p(x)-2}\grad u \cdot \grad\phi\, dx \leq (\geq) 0
\end{align}
for all nonnegative $\phi\in  W^{1,\,p(x)}_0(\Omega)$.
%C^{\infty}_0(\Omega)$.
\end{definition}

\noindent
A function that is both a weak subsolution and a weak supersolution is called a weak solution of the $p(x)$-Laplace equation \eqref{eq:p(x)-lap}. 
A continuous weak solution is called a $p(x)$\emph{-harmonic function}. 
In the above, $C^{\infty}_0(\Omega)$ denotes infinitely differentiable functions with compact support in $\Omega$,
$W^{1,\,p(x)}$ is the \emph{variable exponent Sobolev space} which consists of all functions $u\in L^{p(x)}(\Omega)$ whose weak gradient $\grad u\in L^{p(x)}(\Omega)$. 
For these and further details concerning the variable exponent $p$-Laplace equation, see e.g. \cite{HHLN10}. 

The boundary Harnack inequality in 
 $C^{1,1}$-domains is by now well established for 
$p(x)$-harmonic functions; see, for example, \cite{AL16, AJ17, LOT20}.
In fact, the results in these works ensure that Lemma \ref{le:nonlinear_bhi} remains valid when 
$u_1$ and $u_2$ are assumed to be positive 
$p(x)$-harmonic functions.
It is worth noting that the boundary Harnack inequality can also be derived via Corollary \ref{cor:BHI}, by first rewriting \eqref{eq:p(x)-lap} in the form \eqref{eq:main-nonlinear}, as done in \cite[Section 6]{LOT20}, and then analyzing the resulting ODIs following the approach in \cite[Section 3.3]{L22}.

To establish regularity of the quotient of $p(x)$-harmonic functions we will also use the following $C^{1,\alpha}$ estimate from \cite{AM01}. 

\begin{lemma}\label{le:AM}
Let $r>0$, $p$ be a bounded Lipschitz continuous variable exponent function on $B(0,r)$ and let $u$ be a $p(x)$-harmonic function in $B(0,r)$. 
Then there exist $c>0$ and $\alpha\in(0,1]$ such that 
\begin{equation*}
        |\grad u(x)-\grad u(y)|\leq c|x-y|^{\alpha},\quad |Du(x)|\leq C, \quad \forall (x,y)\in B(0,r/c).
\end{equation*}
\end{lemma}

\noindent
\begin{proof}
This follows from \cite[Theorem 2.2]{AM01}.
\end{proof}

\bigskip

Since the $C^{1,\alpha}$ estimate for $p(x)$-harmonic functions appears to be known only in the interior, 
we may not immediately apply Lemma \ref{le:second} to obtain $u/v \in C^{\alpha}$ near $C^{1,1}$-boundaries.
To be able to apply the interior $C^{1,\alpha}$ estimate over the boundary, we assume that the domain is a half space and extend the $p(x)$-harmonic functions to the whole space via reflection. 
For $y\in\mathbb{R}^n$ and $r\in\mathbb{R}^+$ 
we define the rectangles 
\begin{align} \label{rectangel}
\begin{split}
Q & = \{y:|y_i|<r,\, i\in\{1,...,n\}\}, \\
  Q^+& = \{y:|y_i|<r,\, i\in\{1,...,n-1\},\, 0< y_n <r\}.
\end{split}
\end{align}
We will make use of the following Schwarz reflection principle.

\begin{lemma}\label{schwarz}
Suppose 
that $u$ is $p(x)$-harmonic in $Q^+$, 
continuous on $\overline Q^+$ with $u(x',0) = 0$. 
Define 
\begin{equation*}\label{mirror}
\tilde{u}(x',x_n)=\left\{
	\begin{array}{ll}
		u(x',x_n),\quad\, x_n\geq 0  \\[24pt]
		-u(x',-x_n),\quad\, x_n<0,
	\end{array}
\right. \qquad 
\tilde{p}(x',x_n)=\left\{
	\begin{array}{ll}
		p(x',x_n),\quad\, x_n\geq 0  \\[24pt]
		p(x',-x_n),\quad\, x_n<0.
	\end{array}
\right.
\end{equation*}
%
%\bigskip
\noindent
Then $\tilde{u}$ is $\tilde p(x)$-harmonic in $Q$.
\end{lemma}

\noindent
\begin{proof}
Let $\theta\in C^{\infty}_0\left(Q\right)$ be an arbitrary test function and write $\theta=\phi+\psi$ where 
\begin{equation*}
    \phi(x)=\frac{\theta(x',x_n)+\theta(x',-x_n)}{2}.
\end{equation*}
It follows that $\psi\in W^{1\,\tilde p(x)}_0\left(Q\right)$ and also $\psi(x',0)=0$, 
so that $\psi\in W^{1,\,p(x)}_0(Q^{+})$. 
We have 
\begin{multline}\label{weakp2}
   \int_{Q}|\grad \tilde{u}|^{\tilde{p}(x)-2}\,\langle \grad \tilde{u}, \grad \psi\rangle \, dx 
    =
    \int_{Q^+}|\grad u|^{p(x)-2}\,\langle \grad u , \grad \psi\rangle\, dx \\
    % +\int_{Q_1(0)\setminus Q_1^+(0)}|\grad \tilde{u}|^{\tilde{p}(x)-2}\,\grad \tilde{u} \cdot \grad \psi\, dx\\
   % = 0 
     +\int_{Q\setminus Q^+}|\grad (-u(x',-x_n))|^{p(x',-x_n)-2}\,\langle \grad (-u(x',-x_n)), \grad \psi  \rangle\, dx\\
     =
     -\int_{Q^+}|\grad u|^{p(x',x_n)-2}\,\langle \grad u(x',x_n), \grad \psi(x',-x_n)  \rangle\, dx = 0,
\end{multline}
since $\psi(x',-x_n) \in W^{1,\,p(x)}_0(Q^{+})$.
We also have
\begin{multline}\label{weakp3}
    \int_{Q}|\grad \tilde{u}|^{\tilde{p}(x)-2}\,\langle \grad \tilde{u}, \grad \phi\rangle\, dx
    =\int_{Q^+}|\grad u|^{p(x)-2}\,\langle \grad u, \grad \phi\rangle\, dx \\ 
    + \int_{Q\setminus Q^+}|\grad (-u(x',-x_n))|^{p(x',-x_n)-2}\,\langle \grad (-u(x',-x_n)), \grad \phi\rangle\, dx\\ 
     =
     \int_{Q^+}|\grad u|^{p(x)-2}\,\langle \grad u, \grad \phi\rangle\, dx \\
    - \int_{Q^+}|\grad u|^{p(x',x_n)-2}\,\langle \grad u(x',x_n), \grad \phi(x',-x_n)\rangle\, dx\\
    =
     \int_{Q^+}|\grad u|^{p(x)-2}\,\langle\grad u, \grad \phi\rangle\, dx \\
    - \int_{Q^+}|\grad u|^{p(x',x_n)-2}\,\langle \grad u(x',x_n), \grad \phi(x',x_n)\rangle\, dx = 0.
\end{multline}
From \eqref{weakp2} and \eqref{weakp3} it follows that 
\begin{multline*}\label{weakp3}
\int_{Q}|\grad \tilde{u}|^{\tilde{p}(x)-2}\,\langle\grad \tilde{u}, \grad (\phi+\psi)\rangle\, dx
= \int_{Q}|\grad \tilde{u}|^{\tilde p(x)-2}\,\langle\grad \tilde{u}, \grad \theta\rangle\, dx = 0,
\end{multline*}
which completes the proof.
\end{proof}

%%%%%%%%%%%%%%%%%%%%%%%%%%%%%%%%%%%%%%%%%%%%%

\bigskip

We are now ready to prove the following analogue of Theorem \ref{th:u/v Hölder fully nonlinear} for $p(x)$-harmonic functions.

\begin{theorem}\label{th:p(x)-harmonic}
Let $z\in \partial \mathbb R^n_+$, $1<p^-\leq p(x)\leq p^+<\infty$ be Lipschitz continuous, and let $u,v > 0$ be $p(x)$-harmonic functions in $B(z,r)\cap \mathbb{R}^n_+$,
continuous in $\overline{B(z,r)\cap \mathbb{R}^n_+}$, with $u = 0 = v$ on $B(z,r)\cap \{x : x_n = 0\}$. 
Then there exists $c$ such that
\bigskip

\begin{equation*}%\label{eq:inf-holder-ratio2}
\left| \frac{u(x)}{v(x)} - \frac{u(y)}{v(y)}\right | \leq c  |x-y|^{\alpha}
\end{equation*}
whenever $x,y \in B(z, r/c) \cap \mathbb R^n_+$ and where $\alpha$ is from Lemma \ref{le:AM}. 
Moreover, the result holds with $u$ or $v$ replaced by $d(x, \partial \Omega)$ as well.
\end{theorem}

\noindent
\begin{proof}
Since $u,v$ are $p(x)$-harmonic in $B(z,r)\cap\mathbb{R}^n_{+}$, then by Lemma~\ref{schwarz} $\tilde{u},\tilde{v}$ (defined in~\ref{mirror}) are $\tilde p(x)$-harmonic in $B(z,r)\cap\mathbb{R}^n$. 
Thus, by Lemma \ref{le:AM} there exist $c$ and $\alpha\in(0,1]$ such that
\bigskip
\begin{equation}\label{eq:p(x)_1}
    |\grad \tilde{u}(x)-\grad\tilde{u}(y)| + |\grad \tilde{v}(x)-\grad\tilde{v}(y)|\leq c|x-y|^{\alpha}, \quad |\nabla \tilde u(x)| + |\nabla \tilde v(x)| \leq c
\end{equation}
in $B(z,r/c)\cap\mathbb{R}^n$, wherefore $\tilde{u}, \tilde{v}\in C^{1,\alpha}\left(B(z,r/c)\cap\overline{\mathbb{R}^{n}_{+}}\right)$. 
Moreover, the well known boundary Harnack inequality for $p(x)$-harmonic functions \cite{AL16} %Lemma~\ref{BHIAL} 
gives the existence of $c$ such that
\begin{equation}\label{eq:p(x)_2}
    c^{-1} \leq \frac{y(x)}{d(x,\partial \Omega)} \leq c,\quad\forall\,x\in B(z,r/c)
    %\left(B(z,R)\cap\mathbb{R}^n\right)\cap B(z,R/c).  
\end{equation}
in which $y = u$ or $y = v$. 
Estimates \eqref{eq:p(x)_1} and \eqref{eq:p(x)_2} ensure the assumptions in Lemma~\ref{le:first} and Lemma~\ref{le:second} which therefore give the desired estimates.
\end{proof}

\bigskip

We note that reflection principles for elliptic PDEs and quasiminimizers have been studied, for instance, in \cite{M81, M09}. In addition, $C^{1,\alpha}$ estimates for solutions of PDEs with variable exponents and nonstandard growth conditions are available in, e.g., \cite{S22, V22, BSRR23}.
It is therefore likely that several variants of Theorem \ref{th:p(x)-harmonic} could be derived from Lemma \ref{le:second}, using a combination of known 
$C^{1,\alpha}$ estimates, reflection-based extensions, and boundary Harnack inequalities -- the latter potentially provable via Corollary \ref{cor:BHI}.

%%%%%%%%%%%%%%%%%%%%%%%%%%%%%%%%%%%%%%%%%%%%
%%%%%%%%%%%%%%%%%%%%%%%%%%%%%%%%%%%%%%%%%%%%
%%%%%%%%%%%%%%%%%%%%%%%%%%%%%%%%%%%%%%%%%%%%
%%%%%%%%%%%%%%%%%%%%%%%%%%%%%%%%%%%%%%%%%%%%
%%%%%%%%%%%%%%%%%%%%%%%%%%%%%%%%%%%%%%%%%%%%
%%%%%%%%%%%%%%%%%%%%%%%%%%%%%%%%%%%%%%%%%%%%
%%%%%%%%%%%%%%%%%%%%%%%%%%%%%%%%%%%%%%%%%%%%
%%%%%%%%%%%%%%%%%%%%%%%%%%%%%%%%%%%%%%%%%%%%
%%%%%%%%%%%%%%%%%%%%%%%%%%%%%%%%%%%%%%%%%%%%

\subsection{The $\infty$-Laplace equation}
\label{sec:inf-lap}

\noindent
In this section we derive the analogue of Theorem \ref{th:u/v Hölder fully nonlinear} for solutions of the $\infty$-Laplace equation, 
that is,
\begin{align}\label{eq:inflapequation}
\Delta_{\infty} u := \sum_{i,j = 1}^{n} \frac{\partial u}{\partial x_i} \frac{\partial u}{\partial x_j} \frac{\partial^2 u}{\partial x_i \partial x_j} = 0,
\end{align}
\noindent
which fails to be elliptic in the sense of \eqref{eq:ass_drift_sub} and \eqref{eq:ass_drift_super}.
We adopt the viscosity solution concept as recalled in Section \ref{sec:decay}, that is, \eqref{eq:viscos1} and \eqref{eq:viscos2} with $F = -\Delta_\infty u$, 
and we will refer to a viscosity solution of \eqref{eq:inflapequation} as an $\infty$-harmonic function.

The boundary Harnack inequality for $\infty$-harmonic functions is by now well established in a variety of domains, see, for instance, \cite{Bhatt07,LN08-2,L11} from which we easily conclude the following lemma.

\begin{lemma}\label{le:bhi-infty}
Let $w\in \partial \mathbb R^n_+$, $r > 0$ and $u,v > 0$ be $\infty$-harmonic functions in $B(w,r)\cap \mathbb{R}^n_+$,
continuous on $\overline{B(w,r)\cap \mathbb{R}^n_+}$, with $u = 0 = v$ on $B(w,r)\cap \{x : x_n = 0\}$. 
Then there exists $c$, depending only on $n$, such that
\begin{align*}
 c^{-1} \frac{d(x,\partial \Omega)}{r} \leq \frac{u(x)}{u(a_r(w))} \leq c  \frac{d(x,\partial \Omega)}{r}  \quad \text{and} \quad c^{-1} \frac{u(a_r(w))}{v(a_r(w))} \leq \frac{u(x)}{v(x)} \leq c  \frac{u(a_r(w))}{v(a_r(w))},
\end{align*}
\bigskip
\noindent
whenever $x \in \mathbb{R}^n_+ \cap B(w, r/c)$ and where $a_r(w) = w + (0,\dots, r/2)$.
\end{lemma}

\noindent
Here and in the following, we fix the halfspace to $\mathbb R^n_+$, which is w.l.o.g. since \eqref{eq:inflapequation} is invariant under rotations and translations. 
Notice that, compared to Lemma \ref{le:nonlinear_bhi}, the boundary Harnack inequality above is stronger, owing to its explicit dependence on 
$r$, $u(a_r(w))$ and $v(a_r(w))$. This enhancement stems from the scale invariance and homogeneity of the $\infty$-Laplace equation.

The $C^{1,\alpha}$ regularity of $\infty$-harmonic functions in dimension $n\geq 3$ is a well-known open problem, restricting us to assume $n = 2$. We will make use of the following result, 
which follows from \cite{ES08}.

\begin{lemma}\label{le:Evans-Savin}
Suppose $w \in \mathbb{R}^2$, $r > 0$ and that $u$ is a planar $\infty$-harmonic function in $B(w,r)$. 
Then there are uniform constants $c$ and $\alpha \in (0,1]$ such that
\begin{align*}%\label{eq:C1alpha}
|\nabla u(x) - \nabla u(y) |
\leq c\left(\frac{|x-y|}{r}\right)^\alpha
\sup_{B(w,3r/4)}| \nabla u|
\leq \frac{c}{r}
\left(\frac{|x-y|}{r}\right)^\alpha
\operatorname{osc}_{B(w,r)}u
\end{align*}
whenever $x,y \in B(w,r/2)$.
\end{lemma}

We will also make use of the following well-known Harnack-type inequalities. 

\begin{lemma}\label{le:harnack}
Suppose $w \in \mathbb{R}^n$, $r > 0$ and that $u$ is a positive $\infty$-harmonic function in $B(w,2r)$. 
Then there is $c$, depending only on $n$, such that
\begin{align*}
(i) \quad \sup_{B(w, r)} u \leq c \inf_{B(w, r)} u.
\end{align*}
Moreover, let $w \in \partial \mathbb{R}_+^n$, $r > 0$, and $u$ be a positive $\infty$-harmonic function in $B(w, r) \cap \mathbb{R}_+^n$,
continuous in $B(w, r) \cap \overline{\mathbb{R}_+^n}$ with $u = 0$ on $B(w, r) \cap \partial \mathbb{R}_+^n$.
Then there exists $c$, depending only on $n$, such that
\begin{eqnarray*}
(ii) \quad \sup_{B(w, r/c)} u   \leq   c \,u(a_{r/c}(w)),
\end{eqnarray*}
where $a_{r/c}(w) = w + (0,\dots, r/c)$.
\end{lemma}
\noindent
{\bf Proof.} For the Harnack inequality in $(i)$, see e.g. \cite{lin-man, LN08-2}. 
For $(ii)$,  
a proof for linear elliptic partial differential equations 
in Lipschitz domains can be found in \cite{CFMS81}.
The proof uses only the Harnack chain condition, $(i)$, 
 the well-known Lipschitz continuity up to the boundary as well as the comparison principle in Lemma \ref{jamforelseprin} below.  
In particular, the proof in \cite{CFMS81} applies in our situation. $\hfill \Box$

\bigskip

To make use of the interior $C^{1,\alpha}$ estimate in Lemma \ref{le:Evans-Savin} we extend the $\infty$-harmonic functions over the boundary by Schwarz reflection. 
To state the lemma, recall the notation in \eqref{rectangel} and observe that 
since the infinity Laplace equation is invariant under translations and scalings, it is enough to work with 
$Q^+ := Q^+_1(0)$, $Q^- := Q^-_1(0)$ and $Q := Q_r(w)$. 
We now prove the following. 

\begin{lemma}\label{le:schwarz-inf}
Suppose that $u$ is infinity harmonic in $Q^+$, continuous on $\overline Q^+$ with $u(x',0) = 0$.  
Then the function defined by
\begin{align*}
 \tilde u(x)=\begin{cases}
 u (x) & \mbox{if $\;x_n \geq 0$} \\[\medskipamount]
-u (x', -x_n) & \mbox{if $\;x_n < 0$},
\end{cases}
\end{align*}
is infinity harmonic in $Q$. 
\end{lemma}

\noindent
\begin{proof}
We show the supersolution property (the subsolution property is shown similarly). 
Choose an arbitrary point $\bar x\in Q^-$ and a test function $\psi\in C^2(Q^-)$ such that $\tilde{u}-\psi$ has a local minimum at $\bar x$.
We need to show that 
\begin{align}\label{eq:wändegis}
\bigtriangleup_{\infty} \psi(\bar x) \leq 0.    
\end{align}
We build $\tilde \psi(x)= -\psi(x',-x_n)$ to be a test function, reflected in the same way as $u$. 
Then, by construction, $u(x)-\tilde\psi(x)$ has a local maximum at $(\bar x,-\bar x_n) \in Q^+$. 
Since $u$ is a viscosity subsolution in $Q^+$ we have 
\begin{align}\label{eq:jahapp}
\bigtriangleup_{\infty}\tilde\psi(\bar x',-\bar x_n) \geq 0.    
\end{align}
Now, for $i = 1\dots n-1$, we have
\begin{align*}
\psi_{x_i}(\bar x) = - \tilde \psi_{x_i}(\bar x',-\bar x_n), \quad %\text{and} \quad 
\psi_{x_n}(\bar x) = \tilde \psi_{x_n}(\bar x',-\bar x_n),
\end{align*}
and, for $i, j = 1\dots n-1$, 
\begin{align*}
\psi_{x_i,x_j}(\bar x) &= - \tilde \psi_{x_i,x_j}(\bar x',-\bar x_n),  \quad 
\psi_{x_i,x_n}(\bar x) = \tilde \psi_{x_i,x_n}(\bar x',-\bar x_n),\\
\psi_{x_n,x_n}(\bar x) &= - \tilde \psi_{x_n,x_n}(\bar x',-\bar x_n), 
\end{align*}
we conclude
\begin{align*}
\bigtriangleup_{\infty}\psi(\bar x) 
&=
\sum_{i,j = 1}^{n-1} \psi_{x_i}(\bar x) \psi_{x_j}(\bar x) \psi_{x_i,x_j}(\bar x)\\
&+ 2 \sum_{i = 1}^{n-1} \psi_{x_i}(\bar x) \psi_{x_n}(\bar x) \psi_{x_i,x_n}(\bar x)
+ \psi^2_{x_n}(\bar x)\psi_{x_n,x_n}(\bar x)\\
&=
\sum_{i,j = 1}^{n-1} (-\tilde \psi_{x_i}(\bar x',-\bar x_n))(-\tilde\psi_{x_j}(\bar x',-\bar x_n)) (-\tilde\psi_{x_i,x_j}(\bar x',-\bar x_n))\\
&+ 2 \sum_{i = 1}^{n-1} (-\tilde\psi_{x_i}(\bar x',-\bar x_n) \tilde\psi_{x_n}(\bar x',-\bar x_n) \tilde\psi_{x_i,x_n}(\bar x',-\bar x_n)\\
&+ \tilde\psi^2_{x_n}(\bar x',-\bar x_n)(-\tilde\psi_{x_n,x_n}(\bar x',-\bar x_n))\\
&= -\bigtriangleup_{\infty}\tilde\psi(\bar x',-\bar x_n). 
\end{align*}
Together with \eqref{eq:jahapp} this implies \eqref{eq:wändegis} and therefore we have ensured that $\tilde{u}$ is a viscosity supersolution of \eqref{eq:inflapequation} in $Q \setminus \{x : x_n = 0\}$.

It thus remains to show that $\tilde u$ is a viscosity supersolution on the plane $\{x : x_n = 0\}$.  
Pick $\bar x$ in this plane and let $\psi\in C^2(Q)$ be a test function such that $\tilde u-\psi$ has a local minimum at $\bar x$. 
We need to show that \eqref{eq:wändegis} holds also in this situation. 
Now, since $\tilde u(x',0) = 0$ on the plane, we have, for $i = 1\dots n-1$, 
\begin{align*}
\psi_{x_i}(\bar x) = 0, \quad %\text{and} \quad 
%
%\psi_{x_n}(\bar x) = \tilde \psi_{x_n}(\bar x',-\bar x_n),
\end{align*}
so that 
\begin{align*}
\bigtriangleup_{\infty}\psi(\bar x) =
\psi^2_{x_n}(\bar x)\psi_{x_n,x_n}(\bar x)
\end{align*}
Thus, it suffices to show that $\psi_{x_n,x_n}(\bar x) \leq 0$. 
Suppose, by contradiction, that
$\psi_{x_nx_n}(\bar x)>0$.
Since $\psi\in C^2$, its restriction to the normal line
through $\bar x$ is strictly convex near $\bar x$.
After adding a constant, we may assume that
$\psi(\bar x)=\tilde u(\bar x)=0$.
For sufficiently small $t>0$, strict convexity gives
$$
\psi(\bar x+t e_n)+\psi(\bar x-t e_n) > 2\psi(\bar x) = 0.
$$
On the other hand, since $\psi$ touches $\tilde u$
from below at $\bar x$, odd reflection yields
$$
\psi(\bar x+t e_n)+\psi(\bar x-t e_n)
\leq
\tilde u(\bar x+t e_n)+\tilde u(\bar x-t e_n)=0,
$$
a contradiction. 
Therefore, $\psi_{x_n,x_n}(\bar x) \leq 0$, \eqref{eq:wändegis} holds true and we have proved that $\hat u$ is a viscosity supersolution in $Q$. 
By mimicking the above for the subsolution property we obtain that $\tilde u$ is infinity harmonic in $Q$.
The proof is complete. 
\end{proof}

\bigskip

We are now ready to establish the main result of this subsection, that is the H\"older continuity of quotients of positive planar $\infty$-harmonic functions, vanishing on a hyperplane of dimension $1$.

\begin{theorem}\label{th:holder-boundary-inf-plane}
Let $w\in \partial \mathbb R^2_+$, $r > 0$ and suppose that $u,v > 0$ are $\infty$-harmonic in $B(w,r)\cap \mathbb{R}_+^2$,
continuous on $\overline{B(w,r)\cap \mathbb{R}_+^2}$, with $u = 0 = v$ on $B(w,r)\cap \partial \mathbb R^2_+$. Then there exist $c$ and $\alpha \in (0,1]$ such that
\begin{equation*}%\label{eq:inf-holder-ratio}
\left| \frac{u(x)}{v(x)} - \frac{u(y)}{v(y)}\right | \leq c \frac{u(y)}{v(y)} \left (\frac{ |x-y|}{r}\right)^{\alpha}
\end{equation*}
whenever $x,y \in B(w, r/c) \cap \mathbb R^2_+$. 
\end{theorem}

We remark that the inequality in Theorem \ref{th:holder-boundary-inf-plane} is sometimes written as
$$
\left | \log	{\frac{u(x)}{v(x)}} - \log {\frac{u(y)}{v(y)}}\right | \leq c \left (\frac{ |x-y|}{r}\right)^{\alpha}
$$
and it is stronger than the versions in Theorem \ref{th:u/v Hölder fully nonlinear} and Theorem \ref{th:p(x)-harmonic} as we here have eliminated all dependence in the constants so that $c$ and $\alpha$ are uniform. 
This was possible thanks to the availability of Harnack's inequality, the scale invariance of the infinity Laplace equation and the simple geometry. 
We will give some corollaries of Theorem \ref{th:holder-boundary-inf-plane} after presenting the proof.

\bigskip

\noindent
{\bf Proof of Theorem \ref{th:holder-boundary-inf-plane}}.
By translating, scaling and normalizing we may w.l.o.g. assume $w = 0$, $r = 2$ and $u(e_2) = v(e_2) = 1$. 
Then $u$, $v$ are infinity harmonic in $Q^+$ and, by using Lemma \ref{le:schwarz-inf} we extend $u$, $v$ to be infinity harmonic in $Q$. 
We now apply the interior $C^{1,\alpha}$ estimate in Lemma \ref{le:Evans-Savin} and Lemma \ref{le:harnack} to obtain 
\begin{align}\label{eq:reg_inf}
|\nabla h(x) - \nabla h(y) | &\leq c |x - y|^\alpha \operatorname{osc}_{B(0,1/c)} h 
=
2c |x-y|^\alpha \sup_{B(0,1/c)} h \notag \\ 
&\leq c |x-y|^\alpha \, h(c^{-1} e_n) \leq c |x-y|^\alpha, \\ 
\quad\text{and} \quad |\nabla h(x)| &\leq c \operatorname{osc}_{B(0,1/c)} h \leq c,\notag
\end{align}
whenever $x,y \in B(0,1/c)$ and for $h = u$ and $h = v$. 
Moreover, the boundary Harnack inequality in Lemma \ref{le:bhi-infty} yields
\begin{align}\label{eq:bhi_inf}
c^{-1} x_2 \leq h(x) \leq c x_2 %\quad \text{and} \quad c^{-1} \frac{u(x)}{v(x)} \leq c,
\end{align}
whenever $x,y \in B(0,1/c)$ and for $h = u$ and $h = v$.
Lemma \ref{le:first}, Lemma \ref{le:second}, \eqref{eq:reg_inf} and \eqref{eq:bhi_inf} now imply
\begin{align*}%\label{eq:final-sista-gången}
\left|\frac{u(x)}{v(x)} - \frac{u(y)}{v(y)}\right| \leq  c |x - y|^{\alpha}
\end{align*}
whenever $x,y \in B(0,1/c)$, and where $\alpha$ is from Lemma \ref{le:Evans-Savin}. 
Scaling everything back and using Harnack's inequality and the boundary Harnack inequality once more yield the desired inequality. 
The proof is complete. $\hfill\Box$\\

Next, we derive a boundary gradient estimate stating that $\nabla u \approx u/d$ near the boundary. 
In particular, Theorem \ref{th:holder-boundary-inf-plane} implies the following corollary.

\begin{corollary}\label{cor:grad-inf-est}
Let $w\in \partial \mathbb R^2_+$, $r > 0$ and suppose that $u,v > 0$ are $\infty$-harmonic in $B(w,r)\cap \mathbb{R}_+^2$,
continuous on $\overline{B(w,r)\cap \mathbb{R}_+^2}$, with $u = 0 = v$ on $B(w,r)\cap \partial \mathbb R^2_+$. 
Then there exists $c$ such that
$$
c^{-1}\frac{u(x)}{x_n} \leq |\nabla u| \leq c\frac{u(x)}{x_n} \quad \text{whenever} \quad x \in B(w,r/c)\cap \mathbb{R}_+^2. %Q^+_{1/c}(0).
$$
\end{corollary}

\noindent
\begin{proof}
This follows by the arguments in the proofs of \cite[Lemma 3.18 and Lemma 3.25]{LLuN}.
Indeed, these proofs use only Harnack's inequality, the analogue of Theorem \ref{th:holder-boundary-inf-plane}, and the fact that $x_n$ solves the PDE; thus, the proofs work also in our situation. 
\end{proof}

%%%%%%%%%%%%%%%%%%%%%%%%%%%%%%
%%%%%%%%%%%%%%%%%%%%%%%%%%%%%%
%%%%%%%%%%%%%%%%%%%%%%%%%%%%%%

\subsection{Applications of $u/v \in C^{\alpha}$ in unbounded domains}
\label{sec:applications-of-u/v-reg}

In this section we derive a number of corollaries through scaling and comparison arguments together with certain explicit solutions and the H\"older continuity of the quotients of solutions of PDEs.
We will obtain estimates of $p$-harmonic measures, $p \in (1, \infty]$, as well as versions of the Phragm\'en--Lindel\"of theorem and uniqueness (modulo normalization) of $p$-harmonic functions, including $p = \infty$. 

For $p \in (1,\infty)$, the $p$-Laplace equation yields
\begin{align*}%\label{eq:plapequation}
\Delta_{p} u :=\nabla \cdot ( |\nabla u |^{p - 2} \nabla  u ) = 0,
\end{align*}

\noindent
and we say that $u$ is a \textit{ (weak) subsolution (supersolution)} to the $p$-Laplace equation in
a domain $\Omega$ provided $u \in W_{loc}^{1,p}(\Omega)$ and
\begin{equation*}
\int\limits_{\Omega} | \nabla  u |^{p - 2}  \, \langle   \nabla  u , \nabla \theta  \rangle \, dx \leq (\geq) \,0,
\end{equation*}
whenever $\theta \in C^{\infty}_0(\Omega)$ is non-negative. A function $u$ is a \textit{(weak) solution} of the
$p$-Laplacian if it is both a subsolution and a supersolution.
Here, as in the sequel,
$W^{1,p}(\Omega)$ is the Sobolev space of those $p$-integrable functions whose first
distributional derivatives are also $p$-integrable, 
and $C^\infty_0(\Omega)$ is the set of infinitely differentiable functions with compact support in $\Omega$.
Let $p\in(1,\infty)$. If $u$ is the canonical upper
semicontinuous representative of a weak subsolution of
$\Delta_p u=0$ in $\Omega$, then we say that $u$ is
\textit{$p$-subharmonic} in $\Omega$.
Similarly, if $u$ is the canonical lower semicontinuous
representative of a weak supersolution of $\Delta_p u=0$
in $\Omega$, then we say that $u$ is
\textit{$p$-superharmonic} in $\Omega$.
For $p=\infty$, these notions are defined using upper
semicontinuous viscosity subsolutions and lower
semicontinuous viscosity supersolutions, respectively. 
If $u$ is a continuous solution to the $p$-Laplacian in $\Omega$, $p \in (1, \infty]$, then we say that $u$ is
\textit{$p$-harmonic} in $\Omega$.

To introduce the $p$-harmonic measure, let $\Omega\subset \mathbb{R}^n$ be a regular bounded domain and let $f$ be a real-valued continuous function defined on $\partial{\Omega}$.
It is well known that there exists a unique smooth function $u$, harmonic in $\Omega$,
such that $u = f$ continuously on $\partial \Omega$. The maximum principle and the Riesz representation theorem yield the following representation formula for $u$,
\begin{equation*}
u(z)=\int_{\partial\Omega}f(w) \,{\omega}^{z}(w) \text{, \quad whenever\; $z\in \Omega$.}
\end{equation*}
Here, $\omega^z(w) = \omega(dw, z, \Omega)$ is referred to as the harmonic measure at $z$ associated to the Laplace operator.
As the harmonic measure allows us to solve the Dirichlet problem, its properties are of fundamental
interest in classical potential theory.

In this section we will obtain estimates of the following generalization of harmonic measure related to the $p$-Laplace equation.

%%%%%%%%%%%%%%%%%%%%%%%%%%%%%%%%%%%%%%%%%%%%
%%%%%%%%%%% p-harmonic measure %%%%%%%%%%%%%
%%%%%%%%%%%%%%%%%%%%%%%%%%%%%%%%%%%%%%%%%%%%

%
\begin{definition} \label{def:p-hmeas}
Let $G \subseteq \mathbb{R}^n$ be a domain, $E \subseteq \partial{G}$, $p \in (1,\infty)$
and $x \in G$.
The $p$-harmonic measure of $E$ at $x$ with respect to $G$, denoted by $\omega_p(E,x,G)$,
is defined as $\inf_{u} u(x)$,
where the infimum is taken over all $p$-superharmonic functions $u \ge 0$ in $G$
such that $\liminf_{z \to y} u(z) \ge 1$, for all $y \in E$.
\end{definition}
\noindent
The $\infty$-harmonic measure is defined similarly, but with $p$-superharmonicity replaced by absolutely minimizing, see
%, see Peres--Schramm--Sheffield--Wilson 
\cite[pages 173--174]{inf-tow}.
It turns out that the $p$-harmonic measure in Definition \ref{def:p-hmeas} is a
$p$-harmonic function in $\Omega$, bounded below by $0$ and bounded above by $1$.
For these and other basic properties of $p$-harmonic measure we refer the reader to
%Heinonen--Kilpel\"ainen--Martio 
\cite[Chapter 11]{HKM}.

We will make use of the following well known comparison principle for $p$-harmonic functions, $p \in (1, \infty]$.

%%%%%%%%%%%%%%%%%%%%%%%%%%%%%%%%%%%%%%%%%%%%
%%%%%%%%%%%%%%%%%%%%%%%%%%%%%%%%%%%%%%%%%%%%
%%%%%%%%%%%%%%%%%%%%%%%%%%%%%%%%%%%%%%%%%%%%

\begin{lemma}\label{jamforelseprin}
Let $p \in (1,\infty]$ be given, $u$ be $p$-subharmonic and $v$ be $p$-superharmonic in a bounded domain $\Omega$.
If
\begin{equation*}
\limsup_{x\to y} u(x) \leq \liminf_{x\to y} v(x)
\end{equation*}
for all $y \in \partial \Omega$, and if both sides of the above inequality are not simultaneously $\infty$ or $-\infty$,
then $u \leq v$ in $\Omega$.
\end{lemma}
\noindent
\begin{proof}
If $p \in (1, \infty)$, this result follows from \cite[Theorem~7.6]{HKM}.
For the case $p = \infty$, see e.g. \cite{J93,AS10}.
\end{proof}

\bigskip

In the following subsections, we derive growth estimates for the $p$-harmonic measure, for $p$-subharmonic functions and $p$-harmonic functions in different geometric settings, $p \in (1,\infty]$. We will also consider uniqueness (modulo normalization) of $p$-harmonic functions.

%%%%%%%%%%%%%%%%%%%%%%%%%%%%%%%%%%%%%%%%%%%%
%%%%%%%%%%%%%%%%%%%%%%%%%%%%%%%%%%%%%%%%%%%%
%%%%%%%%%%%%%%%%%%%%%%%%%%%%%%%%%%%%%%%%%%%%

\subsubsection{Infinity-harmonic functions in half-spaces}
\label{sec:inf-spaces-tjohejsan}

With the boundary estimate in Theorem \ref{th:holder-boundary-inf-plane} for $\infty$-harmonic functions at hand,
we can use the explicit solution $u(x)=x_2$ to obtain simple proofs of the probably well-known statements in the following theorem. 
See \cite{BM21} for a similar result on $g$-harmonic functions.
Here and in the following, we use the notation
\begin{align*}
M(R)=\underset{x\in \mathbb R^n}{\underset{|x|=R}{\sup}} u(x)
\end{align*}
and we let $\omega_\infty(x) = \omega_\infty(\mathbb R^2_+ \cap \partial B(w,R),x,\mathbb R^2_+ \cap B(0,R))$ be the $\infty$-harmonic measure of $\mathbb R^2_+ \cap \partial B(0,R)$ at $x=(x_1,x_2)$ w.r.t $\mathbb R^2_+ \cap B(0,R)$.

\begin{theorem}\label{thm:case3}

\vspace{0.3cm}
\textbf{a)} There exists $c$ such that
$$
c^{-1} \frac{x_2}{R} \leq \omega_\infty(x)\leq c  \frac{x_2}{R}, \quad \text{whenever} \quad x \in B(0, R/c) \cap \mathbb R^2_+
$$
%whenever $x \in B(0, R/c) \cap \mathbb R^2_+$.

\vspace{0.3cm}
\textbf{b)}
If $u$ is an $\infty$-subharmonic function in $\mathbb R^2_+$ such that
\begin{align*}
\limsup_{x \to y} u(x) \leq 0 \quad \text{for each}\quad y\in \partial \mathbb{R}_+^2, %y=(x_1,0),
\end{align*}
then either $u \leq 0$ in $\mathbb R^2_+$ or
\begin{align*}
\liminf_{R \to \infty} \frac{M(R)}{R} >0.
\end{align*}

\vspace{0.3cm}
\textbf{c)} Assume that $u$ is a positive $\infty$-harmonic function in $\mathbb R^2_+$, vanishing continuously on $\partial \mathbb{R}_+^2$. 
Then there exists a constant $C$ such that
$$
u(x)= C x_2.
$$
\end{theorem}

%\noindent
%{\bf Proof.}
% $\hfill \Box$\\
%
\noindent
\begin{proof}
\textbf{a)} This is an immediate consequence of the well-known boundary Harnack inequality for $\infty$-harmonic functions recalled in Lemma \ref{le:bhi-infty}.

\vspace{0.3cm}
\textbf{b)} This follows from \textbf{a)} together with a direct application of the Phragm\'en--Lindel\"of principle, see \cite[Thm 11.11, p. 206]{HKM}.

\vspace{0.3cm}
\textbf{c)} 
Let $v(x) = x_2$, and $u$ be an arbitrary positive $\infty$-harmonic function in $\mathbb{R}^2_+$, vanishing on the boundary, and apply Theorem \ref{th:holder-boundary-inf-plane} to obtain the 
existence of uniform constants $c$ and $\alpha \in (0,1]$ such that
\begin{equation*}%\label{eq:inf-holder-ratio}
\left| \frac{u(x)}{x_2} - \frac{u(y)}{y_2}\right | \leq c \,\frac{u(y)}{y_2} \left (\frac{ |x-y|}{R}\right)^{\alpha}
\end{equation*}
whenever $x,y\in B(0, R/c) \cap \mathbb R^2_+$.
Since $R > 0$ is arbitrary, the result follows by letting $R \to \infty$. 
\end{proof}

%%%%%%%%%%%%%%%%%%%%%%%%%%%%%%%%%%%%%%%%%%%%
%%%%%%%%%%%%%%%%%%%%%%%%%%%%%%%%%%%%%%%%%%%%
%%%%%%%%%%%%%%%%%%%%%%%%%%%%%%%%%%%%%%%%%%%%

\subsubsection{$p$-harmonic functions in extended sectors} %in $\mathbb R^n$} 
\label{sec:sectors}

We now consider positive $p$-harmonic functions, $p \in (1,\infty)$, 
vanishing on the boundary of prolonged planar sectors, and thus we generalize some results from \cite{LS23}. 
In particular, consider the domain
\begin{equation} \label{def:Snu}
\mathcal S_\nu =\{ x=(r,\phi, z) := (r,\phi, z_1, \dots, z_{n-2})   \in \mathbb{R}^n : r > 0, |\phi| <\frac{\pi}{2\nu}, z \in\mathbb{R}^{n-2}\},
\end{equation}
where $ \nu\in(1/2, \infty)$. $\mathcal S_\nu$ corresponds to the extension of a $2$-dimensional sector independently in $n-2$ directions.  
Note that $\mathcal S_\nu$ is a Lipschitz domain, and thus 
we may apply the well known H\"older continuity of quotients of $p$-harmonic functions in \cite{LN10}, which implies the following version, tailored for our situation.
Note in particular that since the Lipschitz-constant of $\mathcal S_\nu$ does not change when scaling, and since the $p$-Laplace equation is invariant under scalings, we obtain the freedom to send $R \to \infty$ in the below version.

\begin{lemma}
\label{lemma:Hölder_of_ratio_LN}
Let $p \in (1, \infty)$ and $R > 0$. 
Suppose that $u$ and $v$ are positive $p$-harmonic functions in $\mathcal{S}_\nu \cap B(0, R)$, continuous on $\overline{\mathcal{S}_\nu} \cap B(0, R)$ with $u = v = 0$ on $\partial \mathcal{S}_\nu \cap B(0, R)$. 
Then there exist constants $c$ and $\sigma\in (0,1)$, 
both depending only on $p$, $n$, and $\nu$, such that
$$
\left | \log{\frac{u(x)}{v(x)}} - \log {\frac{u(y)}{v(y)}}\right | \leq c \left ( \frac{|x-y|}{R}\right)^\sigma
$$
whenever $x,y \in \mathcal{S}_\nu \cap B(0, \frac{R}{c})$.
\end{lemma}

\noindent
\begin{proof}
Define $\tilde u(x) = u(Rx)$, $\tilde v(x) = v(Rx)$ and observe that then $\tilde u$ and $\tilde v$ are positive $p$-harmonic functions in $\mathcal{S}_\nu \cap B(0, 1)$, continuous on $\overline{\mathcal{S}_\nu} \cap B(0, 1)$ with $\tilde u = \tilde v = 0$ on $\partial \mathcal{S}_\nu \cap B(0, 1)$.
Since $\mathcal{S}_\nu$ is a Lipschitz domain with Lipschitz constant independent of scaling, and since the $p$-Laplacian is invariant under scalings, we may apply  
\cite[Theorem 2]{LN10}
to obtain constants $c$ and $\sigma \in (0,1)$, depending only on $p$, $n$ and $\nu$, such that
\begin{equation}\label{eq:scaled-proof}
\left | \log{\frac{\tilde u(x)}{\tilde v(x)}} - \log {\frac{\tilde u(y)}{\tilde v(y)}}\right | \leq c \left ( {|x - y|} \right)^\sigma,
\end{equation}
whenever $x,y \in \mathcal{S}_\nu \cap B(0, \frac{1}{c})$.
Let %$y_1, y_2  \in  \mathcal{S}_\nu$ and suppose $R > 0$ is so large that 
$y_1, y_2 \in  \mathcal{S}_\nu \cap B(0,\frac{R}{c})$ and 
define $\tilde y_1 =\frac{y_1}{R}$, 
$\tilde y_2 = \frac{y_2}{R}$ 
so that $\tilde y_1, \tilde y_2 \in  \mathcal{S}_\nu \cap B(0, \frac{1}{c})$. 
By \eqref{eq:scaled-proof} we conclude
\begin{equation*}
\left | \log {\frac{u(y_1)}{v(y_1)}} - \log {\frac{u(y_2)}{ v(y_2)}}\right| 
= \left | \log{\frac{\tilde u(\tilde y_1)}{\tilde v(\tilde y_1)}} - \log {\frac{\tilde u(\tilde y_2)}{\tilde v(\tilde y_2)}}\right | \leq c \left ({|\tilde y_1-\tilde y_2|} \right)^\sigma = c \left ( \frac {| y_1 - y_2|} R \right)^\sigma,
\end{equation*}
which completes the proof. 
\end{proof}

\bigskip
\noindent
We next recall the following explicit solutions of the $p$-Laplace equation in $\mathcal{S}_\nu$;

\begin{lemma} \label{lemma:explicit1}
For any  $\nu \in (\frac{1}{2},\infty)$ and $p \in (1,\infty)$  there exists a function $f_{\nu,p}: [-\frac{\pi}{2\nu},\frac{\pi}{2\nu}] \to \mathbb R $ such that
\begin{equation} \label{eq:explicit1}
h(x) =h(r,\phi, z)= r^{k} f_{\nu,p} (\phi)
\end{equation}
is $p$-harmonic in $\mathcal S_\nu$, where, 
\begin{equation}\label{eq:radialexponent}
k:=k(\nu,p)=\frac{(\nu-1)\sqrt{(1-2\nu)(p-2)^2+\nu^2p^2}+(2-p)(1-2\nu)+\nu^2p}{2(p-1)(2\nu-1)}.
\end{equation}
Moreover, the following properties hold:
\begin{enumerate}
\item $f_{\nu,p}$ is continuous and differentiable on $(-\frac{\pi}{2\nu}, \frac{\pi}{2\nu})$
\item $f_{\nu,p}(\pm \frac{\pi}{2\nu}) = f_{\nu,p}'(0)=0$,
\item $0 \leq f_{\nu,p}(\phi) \leq 1,\, |f_{\nu,p}^\prime(\phi)| \leq c$ \mbox{when $\phi \in [-\frac{\pi}{2\nu}, \frac{\pi}{2\nu}]$}
\item $f_{\nu,p}(\phi) \geq \frac{1}{c}$ \mbox{when $\phi \in [-\frac{\pi}{4\nu}, \frac{\pi}{4\nu}]$,}
$|f_{\nu,p}^\prime(\phi)| \geq \frac{1}{c}$  \mbox{when $\phi \in [-\frac{\pi}{2\nu}, \frac{\pi}{2\nu}] \setminus [-\frac{\pi}{4\nu}, \frac{\pi}{4\nu}] $}
\end{enumerate}
\end{lemma}

\noindent
\begin{proof}
The planar case is summarized in \cite[Lemma 3.1]{LS23} and the extension to $n>2$ is immediate since $h$ is independent of $z$ and the $p$-Laplace equation does not depend explicitly on $n$. Main ideas date back to \cite{K73,A86,P89,BV07} among others. 
\end{proof}

\bigskip

We are now set up to prove our results in the extended sector $\mathcal S_v$. 

\begin{theorem}\label{thm:p-harmonic-measure1}
Assume $p \in (1,\infty)$, $\nu\in (\frac{1}{2}, \infty)$, and $R>0$. Let $\mathcal{S}_\nu \subset \mathbb{R}^n$ be defined by \eqref{def:Snu} and $k:=k(\nu,p)$ be the exponent in~\eqref{eq:radialexponent}. Then the following holds: 

\vspace{0.3cm}
\textbf{a)} There exists $c$, depending only on $\nu$, $p$ and $n$, such that
\begin{align}
c^{-1}\,\left(\frac{r}{R}\right)^{k}\,  \leq \omega_p(x)\, \leq c\, \left(\frac{r}{R}\right)^{k}, \label{eq:mainresult} \qquad \mbox{$x \in \mathcal{S}_{2\nu} \cap B(0,\frac{R}{c})$}
\end{align}
where $\omega_p(x)$ is the $p$-harmonic measure of $\partial B(0,R) \cap \mathcal{S}_\nu$ at $x=(r,\phi,z)$ w.r.t. $\mathcal{S}_\nu \cap B(0,R)$.

\vspace{0.3cm}
\textbf{b)}  If $u(x)$ is a $p$-subharmonic function in $\mathcal S_\nu$ satisfying
\begin{align*}
\limsup_{x \to y} u(x) \leq 0 \quad \text{for each}\; y \in \partial  S_\nu,
\end{align*}
then either $ u \leq 0$ in $ S_\nu$ or
\begin{align*}
\liminf_{R \to \infty} \frac{M(R)}{R^{k}} >0,
\end{align*}
where $k=k(v,p)$ is the exponent in \eqref{eq:radialexponent}, and
\begin{align*}
M(R)=\underset{x\in S_{v}}{\underset{|x|=R}{\sup}} u(x).
\end{align*}

\vspace{0.3cm}
\textbf{c)}
If $u > 0$ is $p$-harmonic in $\mathcal S_\nu$, vanishing continuously on $\partial \mathcal S_\nu$, 
then there exists a constant $C$ such that
$$
u(x) = u(r, \phi, z) = C r^{k} f_{\nu,p} (\phi).
$$
\end{theorem}

\noindent
\begin{proof}
\textbf{a)} Let $u=\omega_p(x)$ be the $p$-harmonic measure and $v=R^{-k} h$ be a scaled version of the explicit solution \eqref{eq:explicit1}. Fix $y = (R/c, 0, 0)$ and apply Lemma \ref{lemma:Hölder_of_ratio_LN} with these functions. After exponentiating, we get
\begin{align}\label{eq:hallojjuss1} 
\frac{\omega_p(x)}{r^{k} R^{-k} f_{\nu,p} (\phi)} \approx \frac{\omega_p (y)}{R^{-k}h(y)}, 
\end{align}
in which constants depend only on $p$, $n$, and $\nu$. 
By the well known Hölder continuity up to the boundary and Harnack's inequality for $p$-harmonic functions, and by the construction of $v$ and the boundary values of $\omega_p(x)$, we also have
\begin{align}\label{eq:hallojjuss2}
\frac{\omega_p (y)}{R^{-k}h(y)} \approx 1,
\end{align}
depending only on $p$, $n$, and $\nu$.
By Lemma \ref{lemma:explicit1} we know that $f_{\nu,p}$ is bounded from above and below in $S_{2\nu}\cap B(0, \frac{R}{c})$. The result therefore follows by \eqref{eq:hallojjuss1} and \eqref{eq:hallojjuss2}.

\vspace{0.3cm}
\textbf{b)} This follows from \textbf{a)} together with a direct application of the Phragm\'en--Lindel\"of principle, see \cite[Thm 11.11, p. 206]{HKM}.

\vspace{0.3cm}
\textbf{c)} Let $h$ be the function specified in Lemma \ref{lemma:explicit1} and let $x,y$ be arbitrary points in $\mathcal{S}_v \cap B(0, \frac{R}{c})$. Applying Lemma \ref{lemma:Hölder_of_ratio_LN} we get
$$
\left| \log\frac{u(x)}{h(x)} - \log \frac{ u(y)}{h(y)} \right| \leq c \left | \frac{x-y}{R}\right |^\sigma.
$$
As $x$ and $y$ were arbitrary and $R$ can be chosen arbitrarily large, we get that $\frac{u}{h}$ must be constant and therefore $u = C h$ for some constant $C$. This concludes the proof.
\end{proof}

Before we proceed to the next geometric setting, we remark that the local estimates in \cite[Section 5]{LS23}, e.g. those near outward/inward cusps, for $p\in (1,\infty)$, 
will also hold in the extended case considered in Theorem \ref{thm:p-harmonic-measure1}.

%%%%%%%%%%%%%%%%%%%%%%%%%%%%%%%%%%%%%%%%%%%%
%%%%%%%%%%%%%%%%%%%%%%%%%%%%%%%%%%%%%%%%%%%%
%%%%%%%%%%%%%%%%%%%%%%%%%%%%%%%%%%%%%%%%%%%%

\subsubsection{$p$-harmonic functions vanishing on an $m$-flat} 
\label{sec:flat}

We finally consider positive $p$-harmonic functions, $p \in (n-m,\infty]$, vanishing continuously on an $m$-flat, i.e., 
an $m$-dimensional hyperplane $\Lambda$, $1\leq m < n$,
\begin{equation} \label{eq:Lambda}
\Lambda = \{x=(x',x'') \in \mathbb{R}^n : |x'|=0 \}, \qquad  x'' \in \mathbb R^{m} \mbox{ and } x' \in \mathbb R^{n-m}.
\end{equation}
To the author's knowledge, few authors have proved estimates of $p$-harmonic functions vanishing on boundaries having dimension less than $n-1$; $m$-flats were considered in \cite{L85} for the borderline case $p = n$, 
and in \cite{L11,L16} similar results were proved for the  general case $p \in (n-m,\infty]$. 
Estimates for solutions of equations of $p$-Laplace type 
near low-dimensional Reifenberg flat sets were proved in \cite{LN18}. 
Indeed, \cite[Theorem 1.9 and Theorem 1.10]{LN18} imply, through the same simple scaling argument as used above, 
an analogue of Lemma \ref{lemma:Hölder_of_ratio_LN} in the current geometric setting. 
The above cited results together with the fact that
\begin{equation} \label{eq:explicit2}
v(x) = |x'|^{\beta} %=d(x,\Lambda)^{\frac{p-n+m}{p-1}} 
\quad \text{with} \quad \beta = \beta(p) = \frac{p-n+m}{p-1}, \quad \beta(\infty) = 1,
\end{equation}
is $p$-harmonic in $\mathbb R^n \setminus \Lambda$, $p \in (1,\infty]$, 
yield the following analogue of Theorem \ref{thm:case3} and Theorem \ref{thm:p-harmonic-measure1}. 

\begin{theorem}\label{thm:case2}
Assume $p \in (n-m, \infty]$, let $\beta$ be as in \eqref{eq:explicit2} and let $\Lambda$ be as in \eqref{eq:Lambda}. 

\vspace{0.3cm}
\textbf{a)} Fix $w \in \Lambda$ and let $\omega_R(x)$ be the $p$-harmonic measure of $\partial B(w,R) \setminus \Lambda$ at $x=(x',x'')$ w.r.t. $B(\omega, R)\setminus \Lambda$. 
Then there exists $c$, depending only on $m$, $n$, and $p$, such that
$$
c^{-1}\left(\frac{d(x,\Lambda)}{R}\right)^{\beta} \leq \omega_p(x)\leq c  \left(\frac{d(x,\Lambda)}{R}\right)^{\beta} \qquad \mbox{$x \in B(w, \frac{R}{c}) \setminus \Lambda$}.
$$

\vspace{0.3cm}
\textbf{b)} If $u(x)$ is $p$-subharmonic in $\mathbb R^n \setminus \Lambda$ satisyfing 
\begin{align*}
\limsup_{x \to y} u(x) \leq 0 \quad \text{for each}\; y \in \Lambda,
\end{align*}
then either $u \leq 0$ in $\mathbb R^n \setminus \Lambda$ or
\begin{align*}
\liminf_{R \to \infty} \frac{M(R)}{R^{\beta}} >0,
\end{align*}
where 
\begin{align*}
M(R)=\underset{x\in \mathbb R^n \setminus \Lambda}{\underset{|x|=R}{\sup}} u(x).
\end{align*}

\vspace{0.3cm}
\textbf{c)} If $p<\infty$ and $u>0$ is $p$-harmonic in
$\mathbb R^n\setminus\Lambda$, vanishing continuously
on $\Lambda$, then, on each connected component of
$\mathbb R^n\setminus\Lambda$, there exists a constant
$C>0$ such that
$$
u(x) = C |x'|^\beta.
$$
\end{theorem}

\noindent
\begin{proof}
Statements \textbf{a)} and \textbf{b)} were proved already in \cite[Section 4]{L16}. 
When $p < \infty$ we can proceed as in the proof of Theorem \ref{thm:p-harmonic-measure1} to obtain statement \textbf{c} as well as alternative proofs of statements \textbf{a)} and \textbf{b)}. 
Indeed, when $p < \infty$ an analogue of Lemma \ref{lemma:Hölder_of_ratio_LN}, for the current geometric setting, 
follows from \cite[Theorem 1.9 and Theorem 1.10]{LN18} by mimicking the proof of Lemma \ref{lemma:Hölder_of_ratio_LN}. 
\end{proof}

\bigskip

\noindent
We remark that the results in Theorem \ref{thm:case2} for $p < \infty$ hold in a more general setting, for $A$-harmonic functions as defined in \cite{LN18}.  
However, this is done at the cost of introducing rather technical assumptions and we therefore refrain from stating this generalization here.

\bigskip

\noindent
{\bf Acknowledgement.} 
The work of Jesper Singh was partially supported by the Swedish research council grant 2018-03743.\\

\noindent
{\bf Data availability statement.}
Our manuscript has no associated data.\\

\noindent
{\bf Conflict of interest statement.}
 On behalf of all authors, the corresponding author states that there is no conflict of interest.\\

%%%%%%%%%%%%%%%%%%%%%%%%%%%%%%%%%%%%%%%%%%%%
%%%%%%%%%%%%%%%%%%%%%%%%%%%%%%%%%%%%%%%%%%%%
%%%%%%%%%%%%%%%%%%%%%%%%%%%%%%%%%%%%%%%%%%%%
%%%%%%%%%%%%%%%%%%%%%%%%%%%%%%%%%%%%%%%%%%%%
%%%%%%%%%%%%%%%%%%%%%%%%%%%%%%%%%%%%%%%%%%%%
%%%%%%%%%%%%%%%%%%%%%%%%%%%%%%%%%%%%%%%%%%%%
%%%%%%%%%%%%%%%%%%%%%%%%%%%%%%%%%%%%%%%%%%%%
%%%%%%%%%%%%%%%%%%%%%%%%%%%%%%%%%%%%%%%%%%%%
%%%%%%%%%%%%%%%%%%%%%%%%%%%%%%%%%%%%%%%%%%%%

\end{document}